\documentclass[12pt]{amsart}

\usepackage[matrix,arrow,curve,frame]{xy}
\usepackage{amsmath,amsthm,amssymb,enumerate}
\usepackage{latexsym}
\usepackage{amscd}
\usepackage[colorlinks=false]{hyperref}
\usepackage{euscript}

\newtheorem{thm}[subsection]{Theorem}
\newtheorem{defn}[subsection]{Definition}
\newtheorem{prop}[subsection]{Proposition}

\newtheorem{cor}[subsection]{Corollary}
\newtheorem{lemma}[subsection]{Lemma}

\newtheorem{remark}[subsection]{Remark}

\theoremstyle{definition}

\numberwithin{equation}{section}

\def\bpartial{{\bar\partial}}

\def\cE{\mathcal E}
\def\cF{\mathcal F}

\def\cI{\mathcal I}
\def\cJ{\mathcal J}

\def\cM{\mathcal M}
\def\cN{\mathcal N}

\def\cS{\mathcal S}
\def\cT{\mathcal T}

\def\cW{\mathcal W}
\def\R{\mathfrak R}

\def\A{\mathfrak A}
\def\B{\mathfrak B}

\def\Zplus{\mathbb Z_{\geq 0}}

\DeclareMathOperator{\sign}{sign}

\DeclareMathOperator{\Hom}{Hom}

\DeclareMathOperator{\Spec}{Spec}

\newfont{\german}{eufm10}

\begin{document}
\pagestyle{plain}

\title{standard monomials and invariant theory for arc spaces III: special linear group}

\author{Andrew R. Linshaw}
\address{Department of Mathematics, University of Denver}
\email{andrew.linshaw@du.edu}
\thanks{A. Linshaw is supported by Simons Foundation Grant \#635650 and NSF Grant DMS-2001484.}

\author{Bailin Song}
\address{School of Mathematical Sciences, University of Science and Technology of China, Hefei, Anhui 230026, P.R. China}
\email{bailinso@ustc.edu.cn}
\thanks{B. Song is supported  by National Natural Science Foundation of China Grant \#12171447.}

\begin{abstract}
This is the third in a series of papers on standard monomial theory and invariant theory of arc spaces. For any algebraically closed field $K$, we prove the arc space analogue of the first and second fundamental theorems of invariant theory for the special linear group. This is more subtle than the results for the general linear and symplectic groups obtained in the first two papers because the arc space of the corresponding affine quotients can be nonreduced.
\end{abstract}


\maketitle
\section{Introduction}

Let $K$ be an algebraically closed field, $G$ be an algebraic group over $K$, and $W$ be a finite-dimensional $G$-module over $K$.  A basic problem of classical invariant theory is to describe the ring of invariant polynomial functions $K[W]^G$. One also considers $K[V]^G$, where $V=W^{\oplus p}\bigoplus {W^*}^{\oplus q}$ is the direct sum of $p$ copies of $W$ and $q$ copies of the dual $G$-module $W^*$. According to Weyl's terminology \cite{W}, a first fundamental theorem of invariant theory (FFT) for the pair $(G,W)$ is a generating set for $K[V]^G$, and a second fundamental theorem (SFT) for $(G,W)$ is a generating set for the ideal of relations among these generators. 

In this paper, $G$ will be the special linear group $SL_h(K)$ over $K$, and $W=K^{\oplus h}$ will be its standard module. Then $V=W^{\oplus p}\bigoplus {W^*}^{\oplus q}$ has affine coordinate ring
$$K[V]=K[a^{(0)}_{il},b^{(0)}_{jl}|\ 0\leq i\leq p,\ 1\leq j\leq p, \ 1\leq l\leq h].$$

\begin{thm} (FFT and SFT for $SL_h(K)$ and $W = K^{\oplus h}$) \label{main:classical}
The invariant ring $K[V]^{SL_h(K)}$ is generated by
\begin{equation} \label{classgenerators} \begin{split} & X^{(0)}_{ij}=\sum_{l=1}^h a^{(0)}_{il}b^{(0)}_{jl},\quad \quad  1\leq i\leq p, \qquad 1\leq j\leq q; 
\\ & Y^{(0)}_{u_1,\dots,u_h}=
\left|
\begin{array}{cccc}
a_{u_1 1}^{(0)} & a_{u_1 2}^{(0)} & \cdots &a_{u_1 h}^{(0)}\\
a_{u_2 1}^{(0)}& a_{u_2 2}^{(0)} & \cdots & a_{u_2 h}^{(0)} \\
\vdots & \vdots & \vdots & \vdots \\
a_{u_h 1}^{(0)} & a_{u_h 2}^{(0)} & \cdots & a_{u_h h}^{(0)} \\
\end{array}
\right|,\quad\quad 
Z^{(0)}_{v_1,\dots,v_h}=\left|
\begin{array}{cccc}
b_{v_1 1}^{(0)} & b_{v_1 2}^{(0)} & \cdots &b_{v_1 h}^{(0)}\\
b_{v_2 1}^{(0)}& b_{v_2 2}^{(0)} & \cdots & b_{v_2 h}^{(0)} \\
\vdots & \vdots & \vdots & \vdots \\
b_{v_h 1}^{(0)} & b_{v_h 2}^{(0)} & \cdots & b_{v_h h}^{(0)} \\
\end{array}
\right|, 
\\ & 1\leq u_1<u_2<\cdots<u_h\leq p,\quad \quad \quad 1\leq v_1<v_2<\cdots<v_h\leq q.
\end{split}
\end{equation}
 Let
$Y^{(0)}_{u,u,u_3,\dots, u_h}=0=Z^{(0)}_{v,v,v_3\dots, v_h}.$ For a permutation $\sigma$ of $1,2,\dots,h$, let
$$Y^{(0)}_{u_{\sigma(1)},\dots,u_{\sigma(h)}}=\sign(\sigma)Y^{(0)}_{u_1,\dots, u_h},\quad Z^{(0)}_{v_{\sigma(1)},\dots,v_{\sigma(h)}}=\sign(\sigma)Z^{(0)}_{v_1,\dots, v_h};$$

 The ideal of relations among the generators is $I_h^+$ generated by 
\begin{equation}\label{eqn:relation0}
 \left|
 \begin{array}{cccc}
 X^{(0)}_{u_1v_1} & X^{(0)}_{u_1v_2} & \cdots &X^{(0)}_{u_1v_h}\\
 X^{(0)}_{u_2v_1}& X^{(0)}_{u_2v_2} & \cdots & X^{(0)}_{u_2v_h} \\
 \vdots & \vdots & \vdots & \vdots \\
 X^{(0)}_{u_hv_1} & X^{(0)}_{u_hv_2} & \cdots & X^{(0)}_{u_hv_h} \\
 \end{array}
 \right|-Y^{(0)}_{u_1,\dots,u_h} Z^{(0)}_{v_1,\dots,v_h},
 \end{equation}
 \begin{equation}\label{eqn:relation1}
 \sum_{i=0}^h (-1)^i
 X^{(0)}_{u_i v}Y^{(0)}_{u_0,\dots,u_{i-1},u_{i+1},\dots,u_h},
 \quad 
 \sum_{j=0}^h (-1)^j
 X^{(0)}_{u v_j}Z^{(0)}_{v_0,\dots,v_{j-1},v_{j+1},\dots,v_h},
 \end{equation}
 \begin{equation}\label{eqn:relation2}
\sum_{i=1}^h(-1)^i
 Y^{(0)}_{u'_{h},\dots,u'_{2},u_i} Y^{(0)}_{u_{1},\dots,u_{i-1},u'_1,u_{i+1},\dots,u_{h}},
 \quad
 \sum_{j=1}^h(-1)^j
 Z^{(0)}_{v'_{h},\dots,v'_{2},v_j} Z^{(0)}_{v_{1},\dots,v_{j-1},v'_1, v_{j+1},\dots,v_{h}}.
 \end{equation}
 \end{thm}

In the case $K = \mathbb{C}$, these theorems are classical and were proven by Weyl \cite{W}. In the general case, the FFT is due to de Concini and Procesi \cite{DCP}, and the SFT is due to Lakshmibai and Shukla \cite{LS}; see also the book \cite{LR} of Lakshmibai and Raghavan for a uniform approach to these theorems for all the classical groups based on standard monomial theory.
 
\subsection{Standard monomial theory} Let 
$$X^{(0)}_{u_1,\dots,u_r;v_1,\dots,v_r}= \left|
\begin{array}{cccc}
	X^{(0)}_{u_1v_1} & X^{(0)}_{u_1v_2} & \cdots &X^{(0)}_{u_1v_r}\\
	X^{(0)}_{u_2v_1}& X^{(0)}_{u_2v_2} & \cdots & X^{(0)}_{u_2v_r} \\
	\vdots & \vdots & \vdots & \vdots \\
	X^{(0)}_{u_rv_1} & X^{(0)}_{u_rv_2} & \cdots & X^{(0)}_{u_rv_r} \\
\end{array}
\right|.$$
There is a partial order for the set 
$$\{ Y^{(0)}_{u_1,\dots,u_h}, Z^{(0)}_{v_1,\dots,v_h}, X^{(0)}_{u_1,\dots,u_r;v_1,\dots,v_r}|\ r\leq h, u_i<u_{i+1}, v_j<v_{j+1}\}$$
given by
\begin{enumerate}
	\item $Y^{(0)}_{u_1,\dots,u_h}\leq Y^{(0)}_{u'_1,\dots,u'_h}$ if $u_i\leq u'_i$ for all $1\leq i\leq h$;
	\item $Z^{(0)}_{v_1,\dots,v_h}\leq Z^{(0)}_{v'_1,\dots,v'_h}$ if $u_i\leq u'_i$ for all $1\leq i\leq h$;
\item $Y^{(0)}_{u_1,\dots,u_h}\leq X^{(0)}_{u'_1,\dots,u'_r;v'_1,\cdots, v'_r}$ if $u_i\leq u'_i$ for all $1\leq i\leq r$;
\item $Z^{(0)}_{v_1,\dots,v_h}\leq X^{(0)}_{u'_1,\dots,u'_r;v'_1,\cdots, v'_r}$ if $v_i\leq v'_i$ for all $1\leq i\leq r$;
\item $X^{(0)}_{u_1,\dots,u_r;v_1,\cdots,v_r}\leq X^{(0)}_{u'_1,\dots,u'_s;v'_1,\cdots, v'_s}$ if $s\leq r$, $u_i\leq u'_i$ and $v_i\leq v'_i$  for all $1\leq i\leq s$.
\end{enumerate}
The invariant ring $K[V]^{SL_h(K)}$ has a standard monomial basis (chapter 11 in \cite{LR}) with respect to this partially ordered set: the ordered products $A_1A_2 \cdots A_k$ of the elements $A_i$  in the partial order set with $A_i \leq  A_{i+1}$ form a basis of $K[V]^{SL_h(K)}$. This is needed in the proof of Theorem \ref{main:classical} in \cite{DCP} and \cite{LS}.

\subsection{Arc spaces}
Our main goal is to prove the arc space analogue of Theorem \ref{main:classical}. Recall that for a scheme $X$ of finite type over $K$, the arc space $J_\infty(X)$ is defined as the inverse limit of the finite jet schemes $J_n(X)$ \cite{EM}. It is generally of infinite type, and is determined by its functor of points by Corollary 1.2 of \cite{B}. For every $K$-algebra $A$, we have a bijection
$$\Hom(\Spec A, J_\infty(X))\cong\Hom(\Spec A[[t]], X).$$ If $f: X\to Y$ is a morphism of schemes, there is an induced morphism of schemes $f_{\infty}:J_{\infty}(X)\to J_{\infty}(Y)$.

If $\text{char} \ K = 0$, $J_{\infty}(X)$ is irreducible whenever $X$ is irreducible, by a theorem of Kolchin \cite{Kol}. However, even if $X$ is irreducible and reduced, $J_{\infty}(X)$ need not be reduced. An ideal $\cI \subset K[J_{\infty}(X)]$ is said to be {\it differentially finitely generated} if there exists a finite subset $\{f_1,\dots, f_m\} \subset \cI$ such that all normalized derivatives $\bpartial^i f_j$ lie in $\cI$, and $\cI$ is generated by $\{\bpartial^i f_j|\ j = 1,\dots, m, \ i \geq 0\}$. If $\text{char} \ K = 0$, the nilradical $\cN\subset K[J_{\infty}(X)]$ is always a differential ideal, that is, $\partial \cN \subseteq \cN$. The question of when $\cN$ is differentially finitely generated has been raised by several authors \cite{BS,SI,SII,KS}. There are examples, such as the homogeneous coordinate ring of the flag varieties in \cite{FM}\cite{M}, but little is known about this question in general. In our main examples, we will give an explicit differential finite generating set for $\cN$ when $\text{char} \ K = 0$.

\subsection{Invariant theory for arc spaces}
Let $G$ be an algebraic group over $K$ and let $V$ be a finite-dimensional $G$-module. Then $J_{\infty}(G)$ is again an algebraic group, and there is an induced action of $J_{\infty}(G)$ on $J_{\infty}(V)$. 
The quotient morphism $V\to V/\!\!/G$ induces a morphism $J_\infty(V)\to J_\infty(V/\!\!/G)$, so we have a morphism \begin{equation} \label{equ:invariantmap} J_\infty(V)/\!\!/J_\infty(G)\to J_\infty(V/\!\!/G).\end{equation}
In particular, we have a ring homomorphism
\begin{equation} \label{equ:invariantringmap} K[J_{\infty}(V/\!\!/G)] \rightarrow K[J_{\infty}(V)]^{J_{\infty}(G)}.\end{equation} In the case $K=\mathbb C$,  if $V/\!\!/G$ is smooth or a complete intersection, it was shown in \cite{LSSI} that under mild hypotheses, \eqref{equ:invariantringmap} is an isomorphism, although in general it is neither injective nor surjective. 

As above, let $W=K^{\oplus h}$ be the standard representation of $SL_h(K)$ and let $V=W^{\oplus p}\bigoplus {W^*}^{\oplus q}$. Then $J_{\infty}(V)$ has affine coordinate ring
\begin{equation} \label{coordringJV} K[J_{\infty}(V)] = K[a^{(k)}_{il},b^{(k)}_{jl}|\ 1\leq i\leq p, \ 1\leq j\leq q, \ 1\leq l\leq h, \ k\in \Zplus ],\end{equation} which has an induced action of $J_\infty(SL_h(K))$.

Our main result is the following.
  \begin{thm}\label{thm:JSLinvariant+} Let
  	$K[J_{\infty}(V)]^{J_\infty({SL}_h(K))}$ be the ring of invariants under the action of $J_\infty(SL_h(K))$ on $K[J_{\infty}(V)]$.
 	\begin{enumerate}
 		\item 
 	$K[J_{\infty}(V)]^{J_\infty({SL}_h(K))}$ is generated by 
 		\begin{eqnarray*}
 			\bpartial^k X^{(0)}_{ij}, \quad
 			\bpartial^k 	Y^{(0)}_{u_1,\dots,u_h},\quad
 		\bpartial^k	Z^{(0)}_{v_1,\dots, v_h},\quad  k\in\Zplus
 		\end{eqnarray*}
 	with  $1\leq i\leq p, \ 1\leq j\leq q, \ 1\leq u_1<u_2<\cdots<u_h\leq p, \ 1\leq v_1<v_2<\cdots<v_h\leq q$.
 		\item 
 		The ideal of relations among the generators in (1) is generated by 
 		\begin{equation}\label{eqn:relationinvariant0}
 		\bpartial^n\left|
 		\begin{array}{cccc}
 		X^{(0)}_{u_1v_1} & X^{(0)}_{u_1v_2} & \cdots &X^{(0)}_{u_1v_h}\\
 		X^{(0)}_{u_2v_1}& X^{(0)}_{u_2v_2} & \cdots & X^{(0)}_{u_2v_h} \\
 		\vdots & \vdots & \vdots & \vdots \\
 		X^{(0)}_{u_hv_1} & X^{(0)}_{u_hv_2} & \cdots & X^{(0)}_{u_hv_h} \\
 		\end{array}
 		\right|-\sum_{k=0}^n\bpartial^k Y^{(0)}_{u_1,\dots,u_h} \bpartial^{n-k} Z^{(0)}_{v_1,\dots,v_h},\qquad n \geq 0,
 		\end{equation}

 		\begin{equation}\label{eqn:relationinvariant1} \begin{split}
 		& \sum_{k=0}^n\sum_{i=0}^h (-1)^i
 		\bpartial^k X^{(0)}_{u_i v}\bpartial^{n-k} Y^{(0)}_{u_0,\dots,u_{i-1},u_{i+1},\dots,u_h},\qquad n \geq 0,
		\\ &	
 		\sum_{k=0}^n\sum_{j=0}^h (-1)^j
 		\bpartial^k X^{(0)}_{u v_j}\bpartial^{n-k} Z^{(0)}_{v_0,\dots,v_{j-1},v_{j+1},\dots,v_h},\qquad n\geq 0,\end{split} \end{equation}
 		
 		and for any $0\leq l<i \leq h$,
 		\begin{equation}\label{eqn:relationinvariant2}
 		\sum_{k=0}^n\sum_{\sigma} C_{k}^l\frac{\sign(\sigma)}{h!i!} 
 		\bpartial^{n-k} Y^{(0)}_{u'_{h},\dots,u'_{i+1},u_{\sigma(h+i)},\dots,u_{\sigma(h+1)}} \bpartial^k Y^{(0)}_{u_{\sigma(h)},\dots,u_{\sigma(1)}};
 		\end{equation}
 		\begin{equation}\label{eqn:relationinvariant3}
 		\sum_{k=0}^n\sum_{\sigma} C_{k}^l\frac{\sign(\sigma)}{h!i!} 
 		\bpartial^{n-k} Z^{(0)}_{v'_{h},\dots,v'_{i+1},v_{\sigma(h+i)},\dots,v_{\sigma(h+1)}} \bpartial^{k} Z^{(0)}_{v_{\sigma(h)},\dots,v_{\sigma(1)}}.
 		\end{equation}
 		Here the summations are over all permutations $\sigma$ of $1,2,\dots, h+i$, and \begin{equation} \label{def:ckn} C_k^l=\left\{ \begin{array}{cc}
\frac{k!}{l!(k-l)!}, &0\leq l\leq k,\\
0, &\text{otherwise.}
\end{array}\right.
\end{equation}
 		\item There is a  standard monomial basis for $K[J_{\infty}(V)]^{J_\infty(SL_h(K))}$.
 	\end{enumerate}
 \end{thm}
This generalizes the results in Section 4 of \cite{LSSI}, which deal with the following special cases for $K = \mathbb{C}$.
\begin{enumerate}
\item $p<n$ and $q \leq n$, or $p = n$ and $q<n$, so that $V/\!\!/SL_h(\mathbb C)$ is an affine space,
\item $p = n = q$, so that $V/\!\!/SL_h(\mathbb C)$ is a hypersurface,
\item $p = n+1$ and $q<n$, or $p <n$ and $q = n+1$, so that $V/\!\!/SL_h(\mathbb C)$ is a complete intersection.
\end{enumerate}
The first statement in Theorem \ref{thm:JSLinvariant+} shows that the map \eqref{equ:invariantringmap} is surjective. 
Relations (\ref{eqn:relationinvariant0}) and (\ref{eqn:relationinvariant1}) are obtained from relations (\ref{eqn:relation0}) and (\ref{eqn:relation1}), and relations (\ref{eqn:relationinvariant2}) and (\ref{eqn:relationinvariant3}) for $l=0$
 are obtained from relations (\ref{eqn:relation2}), by taking $\bpartial^n$. However, relations (\ref{eqn:relationinvariant2}) and (\ref{eqn:relationinvariant3}) for $0 < l < h$ are not consequences of (\ref{eqn:relation0}), (\ref{eqn:relation1}), and (\ref{eqn:relation2}), and their derivatives. Therefore \eqref{equ:invariantringmap} is not injective if $\text{max}(p,q) -2 > h$. We will see later that when $\text{char} \ K = 0$, the nilradical $\cN$ in the affine coordinate ring $K[J_{\infty}(V /\!\!/ SL_h(K)]$ coincides with the kernel of \eqref{equ:invariantringmap}. Therefore $\cN$ is generated by the relations (\ref{eqn:relationinvariant2}) and (\ref{eqn:relationinvariant3}) in the case $n=0$ and $0<l < h$, together with their normalized derivatives; in particular, $\cN$ is differentially finitely generated. When $p=0$ or $q=0$, only the relations  (\ref{eqn:relationinvariant2}) and (\ref{eqn:relationinvariant3}) appear respectively; they are the semi-infinite Pl\"ucker relations in \cite{FM} for the semi-infinite Grassmannians.
 
The main idea in the proof is to extend the standard monomial basis for $K[V]^{SL_h(K)}$ to a standard monomial basis for $K[J_{\infty}(V)]^{J_\infty({SL}_h(K))}$. In Section \ref{sec:standardmonomials}, we will introduce an integral form of this ring which we denote by $\R^+_h$. It is the quotient of the polynomial ring over $\mathbb{Z}$ generated by $x^{(k)}_{ij}$, $\bpartial^k (u_h, \dots, u_2, u_1|$, and $\bpartial^k | v_1, v_2,\dots, v_h)$, by the ideal generated by the relations \eqref{eqn:relationinvariant0}-\eqref{eqn:relationinvariant3}, where we replace the variables $\bpartial^k X^{(0)}_{ij}$, $\bpartial^k Y^{(0)}_{u_1,\dots, u_h}$, and $\bpartial^k Z^{(0)}_{v_1,\dots, v_h}$ with $x^{(k)}_{ij}$, $\bpartial^k (u_h, \dots, u_2, u_1|$ and $\bpartial^k | v_1, v_2,\dots, v_h)$, respectively. We will define the set of standard monomials in $\R^+_h$, and show (modulo the technical Lemma \ref{lem:base20}, whose proof will be given in Section \ref{sec:base}), that the standard monomials span $\R^+_h$ over $\mathbb{Z}$.

In Section \ref{sec:canonicalbasis}, we will prove that the standard monomials form a $\mathbb{Z}$-basis for $\R^+_h$;  see Theorem \ref{thm:standmonomial1+}. The proof makes use of a ring $\B$ which is an integral form of $K[J_{\infty}(V)]$, i.e., $K[J_{\infty}(V)] \cong \B \otimes_{\mathbb{Z}} K$, and a homomorphism $Q^+_h: \R^+_h \rightarrow \B$ which we show is injective. In Section \ref{sec:application}, we will show that after tensoring with $K$, the image of $Q^+_h$ can be identified with $K[J_{\infty}(V)]^{J_\infty({SL}_h(K))}$; see Theorem \ref{thm:JSLinvariant}. We will also show that $\R^+_h$ is the quotient of a larger ring $\tilde{\R}^+_h$ which is an integral form of $K[J_{\infty}(V /\!\!/ SL_h(K)]$. So after tensoring with $K$, the composition of the maps $\tilde{\R}^+_h \rightarrow \R^+_h$ and $Q^+_h: \R^+_h \rightarrow \B$ is exactly the natural map $K[J_{\infty}(V /\!\!/ SL_h(K)] \rightarrow K[J_{\infty}(V)]^{J_\infty({SL}_h(K))}$, which proves Theorem \ref{thm:JSLinvariant+}. Finally, Section \ref{sec:properties} is devoted to establishing some technical properties of standard monomials which are needed in the proof of Lemma \ref{lem:base20}.

Theorem \ref{thm:JSLinvariant+} has significant applications to vertex algebras which we develop in \cite{LS2,LS3}. First, it allows a complete description of certain cosets of affine vertex algebras inside free field algebras that are related to the classical Howe pairs. These cosets appear in some new level-rank dualities involving affine vertex superalgebras. Next, for any smooth manifold $X$ in either the algebraic, complex analytic or smooth settings, the {\it chiral de Rham complex}  $\Omega^{\text{ch}}_X$ is a sheaf of vertex algebras on $X$. It was introduced by Malikov, Schechtman and Vaintrob in \cite{MSV}, and has attracted considerable attention in both physics and mathematics. Theorem \ref{thm:JSLinvariant+} is essential in the description of the vertex algebra of global sections $\Gamma(X, \Omega^{\text{ch}}_X)$ for a $d$-dimensional compact K\"ahler manifold $X$ with holonomy group $SU(d)$. This algebra turns out to be isomorphic to the algebra with central charge $3d$ constructed by Odake in \cite{O}. It is an important building block in the structure of $\Gamma(X, \Omega^{\text{ch}}_X)$ for an arbitrary compact Ricci-flat K\"ahler manifold $X$.

\section{Standard monomials} \label{sec:standardmonomials}
\subsection{A ring and its derivation} Fix integers $p,q \geq 1$, and let 
$$\R = \R^{p,q}=\mathbb Z[x^{(k)}_{ij}| \ 1\leq i\leq p,\ 1\leq j\leq q,\ k\in \Zplus].$$ This is related to the affine coordinate ring of the arc space of the space $M_{p,q} = M_{p,q}(K)$ of $p\times q$ matrices over a field $K$ as follow: $K[J_{\infty}(M_{p,q})] \cong \R \otimes_{\mathbb{Z}} K$.

For $h>0$, let $\cJ_h^L$ be the set of expressions
 \begin{equation}\label{eqn:Lsequence}J^L=\bpartial^k(u_h,\dots,u_2,u_1|,\quad \text{with } 1\leq u_1< u_2<\cdots <u_h\leq p,\, k\in \mathbb Z_{\geq 0},
\end{equation}
 and $\cJ_h^R$ be the set of expressions \begin{equation}\label{eqn:Lsequence}J^R=\bpartial^k|v_1,v_2,\dots,v_h),\quad \text{with } 1\leq v_1< v_2<\cdots <v_h\leq q,\, k\in \mathbb Z_{\geq 0}. \end{equation} We call $wt(J^L)=wt(J^R)=k$  the weights of $J^L$ and $J^R$.

For a permutation $\sigma$ of $1,2,\dots ,h$, let $$\bpartial^k|u_{\sigma(h)},\dots, u_{\sigma(2)},u_{\sigma(1)})=\sign(\sigma)J^L,\quad \bpartial^k|v_{\sigma(1)},v_{\sigma(2)},\dots,v_{\sigma(h)})=\sign(\sigma)J^R.$$
If there is $1\leq i< j\leq h$ with $u_i=u_j$, let  $\bpartial^k(u_h,\dots,u_2,u_1|=0$.
If there is $1\leq i< j\leq h$ with $v_i=v_j$, let
$\bpartial^k|v_1,v_2,\dots,v_h)=0$.

There is a derivation $\partial$ on the polynomial ring $\R^+=\R[ \cJ_h^L\cup  \cJ_h^R]$ given by 
\begin{eqnarray*}
\partial x^{(k)}_{ij}&=&(k+1)x^{(k+1)}_{ij},\\
\partial\,\bpartial^k(u_h,\dots,u_2,u_1|&=&
(k+1)\bpartial^{k+1}(u_h,\dots,u_2,u_1|,\\
\partial\,\bpartial^k|v_1,v_2,\dots,v_h)&=&(k+1)\bpartial^{k+1}|v_1,v_2,\dots,v_h).
\end{eqnarray*}
Let $\bpartial^k=\frac 1{k!} \partial^k$ be the normalized $k$-derivation. The following propositions are easy to verify.
\begin{prop} \label{prop:dab} For any $a,b\in\R^+$,
	$$\bpartial^l(ab)=\sum_{i=0}^l\bpartial^ia\, \bpartial^{l-i}b$$
	and $\bpartial^l a\in \R^+$.
\end{prop}
Let
\begin{equation}\label{eqn:minor} B=\left|
\begin{array}{cccc}
x^{(0)}_{u_1v_1} & x^{(0)}_{u_1v_2} & \cdots &x^{(0)}_{u_1v_h}\\
x^{(0)}_{u_2v_1}& x^{(0)}_{u_2v_2} & \cdots & x^{(0)}_{u_2v_h} \\
\vdots & \vdots & \vdots & \vdots \\
x^{(0)}_{u_hv_1} & x^{(0)}_{u_hv_2} & \cdots & x^{(0)}_{u_hv_h} \\
\end{array}
\right|.
\end{equation}
\begin{prop}\label{prop:detexpression}

	\begin{equation}\bpartial^n B=\sum_{\substack{n_1+\cdots+n_h=n\\n_i\in \mathbb Z_{\geq 0}}}\sum_{\sigma}\sign(\sigma)x^{(n_1)}_{u_1v_{\sigma(1)}} x^{(n_2)}_{u_2v_{\sigma(2)}}\cdots x^{(n_h)}_{u_hv_{\sigma(h)}}.
	\end{equation}
\end{prop}

Let 
\begin{equation} \label{def:ringrplush} \R^+_h=\R^+/\cI_h^+,\end{equation} where $\cI_h^+$ is the ideal generated by
	\begin{equation}\label{eqn:relationLR}
	\bpartial^n B-\bpartial^n(\bpartial^0(u_h,\dots,u_2,u_1| \,\,\,\bpartial^{0}|v_1,v_2,\dots,v_h))=0.
	\end{equation}	
	\begin{equation}\label{eqn:basicrelation2}\bpartial^n\sum_{i=0}^h (-1)^i
x^{(0)}_{u_i v}\,	\bpartial^{0} (u_h,\dots,u_{i+1},u_{i-1},\dots,u_0|
	=0;
	\end{equation}
	\begin{equation}
	\label{eqn:baiscrelation3}
	 \bpartial^n\sum_{j=0}^h (-1)^j
	x^{(0)}_{u v_j}\,	\bpartial^{0}|v_0,\dots,v_{j-1},v_{j+1},\dots,v_h)
	=0;
	\end{equation}
 For any $0\leq l<i$,
	\begin{equation}\label{eqn:basicrelation4}
	\bpartial^n\sum_{\sigma} \frac{\sign(\sigma)}{h!i!} 
	\bpartial^{0} (u'_{h},\dots,u'_{i+1},u_{\sigma(h+i)},\dots,u_{\sigma(h+1)}|\,\,	\bpartial^{l} (u_{\sigma(h)},\dots,u_{\sigma(1)}|
	=0;
	\end{equation}
	\begin{equation}\label{eqn:basicrelation5}
	\bpartial^n\sum_{\sigma} \frac{\sign(\sigma)}{h!i!} 
	\bpartial^{0} |v'_{h},\dots,v'_{i+1},v_{\sigma(h+i)},\dots,v_{\sigma(h+1)})\,\,	\bpartial^{l} |v_{\sigma(h)},\dots,v_{\sigma(1)})
	=0.
	\end{equation}
Here the summations are over all permutations $\sigma$ of $1,2,\dots, h+i$. Then $\R^+_h$ inherits the derivation $\partial$ and normalized $k$-derivation $\bpartial^k$ from $\R^+$.

\subsection{Generators}
The determinant $B$ in \eqref{eqn:minor} can be represented by the pair of ordered $h$-tuples 
$$(u_h,\dots,u_2,u_1|v_1,v_2,\dots,v_h),\qquad 1\leq u_i< u_{i+1}\leq p,\qquad 1\leq v_i< v_{i+1}\leq q.$$ Similarly, let
\begin{equation}\label{eqn:sequenceJ}
J=\bpartial^k(u_h,\dots,u_2,u_1|v_1,v_2,\dots,v_h),\qquad 1\leq u_i< u_{i+1}\leq p,\qquad 1\leq v_i< v_{i+1}\leq q,
\end{equation} 
represent $\bpartial^kB\in\R$. For convenience, we will call such expressions {\it $\bpartial$-lists} throughout this paper. We call $wt(J)=k$  the weight of $J$ and call $sz(J)=h$ the size of $J$.  Let $\cJ$ be the set of these $\bpartial$-lists, and
$$\cJ_h=\{J\in \cJ| \ sz(J)\leq h\}$$ 
be the set of elements in $\cJ$ with size less than or equal to $h$.  Let 
$$\cJ_h^+= \cJ_{h-1} \cup \cJ_h^L\cup  \cJ_h^R.$$
$\R_h^+$ is generated by $\cJ_h^+$. 
Let $$J(L)=\bpartial^0(u_{h},\dots,u_1|\in\cJ^L_{h}, \quad J(R)=\bpartial^0|v_1,\dots,v_{h})\in \cJ^R_{h'}$$
be the left hand part and the right hand part of $J$, respectively.

Let $\cE$ be the set of pairs of ordered $h$-tuples of ordered pairs
\begin{equation}\label{eqn:element}
E=((u_h,k_h),\dots,(u_2,k_2),(u_1,k_1)|(v_1,l_1),(v_2,l_2),\dots,(v_h,l_h)),
\end{equation}
where
\begin{enumerate}
\item $1\leq u_i\leq p$, and $u_i\neq u_j$ if $i\neq j$,
\item $1\leq v_i\leq q$, and $v_i\neq v_j$ if $i\neq j$,
\item $k_i,l_j\in \Zplus$.
\end{enumerate}
Let
$$
||E||=\frac{1}{n!}\partial^n(u_{\sigma(h)},\dots,u_{\sigma(2)},u_{\sigma(1)}|v_{\sigma'(1)},v_{\sigma'(2)},\dots,v_{\sigma'(h)})\in\cJ.
$$
Here
$n=\sum k_i+\sum l_i$ and $\sigma$, $\sigma'$ are the unique permutations of $1,2,\dots, h$ such that $u_{\sigma(i)}< u_{\sigma(i+1)}$ and $v_{\sigma(i)}<v_{\sigma(i+1)}$. Let 
$$wt(E)=wt(||E||),\quad\quad sz(E)=sz(||E||).$$
Let $$\cE_h=\{E\in\cE|\,\,sz(E)\leq h\}.$$
For $h>0$,
let $\cE_h^L$ be the set of ordered $h$-tuples of ordered pairs
\begin{equation}\label{eqn:LEsequence}
E^L=((u_h,k_h),\dots,(u_2,k_2),(u_1,k_1)|,\quad 1\leq u_i\leq p,\, k_i\geq \Zplus
\end{equation}
with $1\leq u_i\leq p$ and $u_i\neq u_j$ if $i\neq j$. Similarly, let
$\cE_h^R$ be the set of ordered $h$-tuples of ordered pairs
\begin{equation}\label{eqn:REsequence}
E^R=|(v_1,l_1),(v_2,l_2),\dots,(v_h,l_h)),\quad 1\leq v_j\leq q,\, l_j\geq \Zplus
\end{equation}
with $1\leq v_i\leq q$ and $v_i\neq v_j$ if $i\neq j$.
Let
$$||E^L||=\bpartial^k(u_{\sigma(h)},\dots,u_{\sigma(2)},u_{\sigma(1)}|\in \cJ^L_h, \quad ||E^R||=\bpartial^k|v_{\sigma'(1)},v_{\sigma'(h)},\dots,v_{\sigma'(h)})\in\cJ^R_h.$$
Here $k=\sum k_i$, $l=\sum k_j$, $\sigma$ and $\sigma'$ are the unique permutations of $1,2,\dots,h$ such that 
$u_{\sigma(i)}<u_{\sigma(i+1)}$ and $v_{\sigma'(i)}<v_{\sigma'(i+1)}$.
Let $$wt(E^L)=wt(||E^L||), \quad wt(E^R)=wt(||E^R||).$$
Let $\cE_h^+= \cE_h^L\cup \cE_h^R\cup\cE_{h-1}$. For $J\in \tilde \cJ_h^+$, let 
\begin{equation} \label{def:E(J)} \cE(J)=\{E \in \cE_h^+|\,\, ||E||=J\}.\end{equation}
We can use the elements in $\cE(J)$ to represent $J$.

Let 
$$F:\cE^L_h\times \cE^R_h\to \cE,\quad F(E^L,E^R)=E.$$
Let
$E(L)=E^L$ and $E(R)=E^R$ be the left hand part and the right hand part of $E$, respectively. 
Then for any $E\in \cE$, $F(E(L), E(R))=E$.

\subsection{Ordering} For any set $\cS$, let $\cM(\cS)$ be the set of ordered products of elements of $\cS$.
If $\cS$ is an ordered set, we order the set  $\cM(\cS)$ lexicographically, that is 
$$S_1S_2\cdots S_m\prec S'_1S'_2\cdots S'_n$$ 
if for some $i_0 \leq m$ we have $S_i=S'_i$ for all $i< i_0$, and $S_{i_0}\prec S'_{i_0}$, or if $S_i=S'_i$ for all $i \leq m$ and $n>m$. We order $\cM(\mathbb Z)$, the set of ordered products of integers, lexicographically.

 We define an ordering for the set $ \cJ_h^+$:
\begin{enumerate}
	\item $
	\bpartial^k(u_{r},\dots,u_2,u_1|v_1,v_2,\dots,v_{r})\prec \bpartial^{k'}(u'_{h'},\dots,u'_2,u'_1|v'_1,v'_2,\dots,v'_{h'})
	$, if $h'< r$, or $h'=r$ and $k< k'$, or  $h'=r$, $k= k'$  and $u_{r}\cdots u_1v_{r}\cdots v_1\prec u'_{r}\cdots   u'_1v'_{r}\cdots v'_1$.
	\item
	$\bpartial^k(u_h,\dots,u_2,u_1| \prec \bpartial^{k'} |v_1,v_2,\dots,v_h)\prec \bpartial^{k''}(u'_{h'},\dots,u'_2,u'_1|v'_1,v'_2,\dots,v'_{h'})$,
	\item
	$\bpartial^k(u_h,\dots,u_2,u_1|\prec \bpartial^{k'}(u'_h,\dots,u'_2,u'_1|$, if
	$k< k'$, or  $k= k'$  and $u_h\cdots u_1\prec u'_h\cdots u'_1$;\\
	$\bpartial^{k} |v_1,v_2,\dots,v_h)\prec \bpartial^{k'} |v'_1,v'_2,\dots,v'_h)$, if
	$k< k'$, or  $k= k'$  and 
 $v_h\cdots v_1\prec v'_h\cdots v'_1$. 
\end{enumerate}
We order $\cM(J_h^+)$, the words of elements of $\cJ_h^+$ lexicographically. 

 We order the pairs $(u,h)\in \Zplus\times \Zplus$ by
$$(u,h)\leq (u',h'), \text{ if } h<h' \text{ or } h=h' \text{ and }u\leq u'.$$  We order $\cM(\Zplus\times \Zplus)$, the words of $\Zplus\times \Zplus$ lexicographically.

There is an ordering for the set $\cE_h^+$:
\begin{enumerate}
	\item $ 
	((u_r,k_r),\dots,(u_1,k_1)|(v_1,l_1),\dots,(v_r,l_r))\prec ((u'_{h'},k'_{h'}),\dots,(u'_1,k'_1)|(v'_1,l'_1),\dots,(v'_{h'},l'_{h'}))
	$, if $r>h'$, or $r=h'$ and $\sum (k_i+l_i)<\sum (k'_i+l'_i)$, or  $r=h'$, $\sum (k_i+l_i)=\sum (k'_i+l'_i)$ and
		$(u_r,k_r)\cdots(u_1,k_1)(v_r,l_r)\cdots(v_1,l_1)\prec (u'_{h'},k'_{h'})\cdots(u'_1,k'_1)(v'_{h'},l'_{h'})\cdots(v'_{1},l'_{1})$,
	\item $((u_h,k_h),\dots,(u_1,k_1)|\prec |(v_1,l_1),\dots,(v_h,l_h))\prec ((u'_{h'},k'_{h'}),\dots,(u'_1,k'_1)|(v'_1,l'_1),\dots,(v'_{h'},l'_{h'}))$,
	\item  $((u_h,k_h),\dots,(u_1,k_1)|\prec ((u'_h,k'_h),\dots,(u'_1,k'_1)|$\\
	if
	$\sum k_i<\sum k'_i$ or $\sum k_i=\sum k'_i$ and
	$(u_h,k_h)\cdots(u_1,k_1)\prec (u'_{h},k'_{h})\cdots(u'_1,k'_1)$,
	\item
	$|(v_1,l_1),\dots,(v_h,l_h))\prec |(v'_1,l'_1),\dots,(v'_h,l'_h))$\\
	if  $\sum l_i<\sum l'_i$ or $\sum l_i=\sum l'_i$ and $(v_h,l_h)\cdots(v_1,l_1)\prec (v'_{h},l'_{h})\cdots (v'_{1},l'_{1})$.
\end{enumerate}
There is a partial ordering for the set $\cE_h^+$:
\begin{enumerate}
	\item $
	((u_r,k_r),\dots,(u_1,k_1)|(v_1,l_1),\dots,(v_r,l_r))\leq ((u'_{h'},k'_{h'}),\dots,(u'_1,k'_1)|(v'_1,l'_1),\dots,(v'_{h'},l'_{h'}))
	$, 
	if $h'\leq r$ and $(u_i,k_i)\leq (u'_i,k'_i)$, $(v_i,l_i)\leq (v'_i,l'_i)$, for $1\leq i\leq h'$,
	\item $((u_{h},k_{h}),\dots,(u_1,k_1)|\leq ((u'_{h'},k'_{h'}),\dots,(u'_1,k'_1)|(v'_1,l'_1),\dots,(v'_{h'},l'_{h'}))$\\
	if  $(u_i,k_i)\leq (u'_i,k'_i)$, for $1\leq i\leq h'$;
	\item
	$|(v_1,l_1),\dots,(v_{h},l_{h}))\leq ((u'_{h'},k'_{h'}),\dots,(u'_1,k'_1)|(v'_1,l'_1),\dots,(v'_{h'},l'_{h'}))$\\
	if $(v_i,l_i)\leq (v'_i,l'_i)$, for $1\leq i\leq h'$;
	\item $ ((u_{h},k_{h}),\dots,(u_1,k_1)|\leq ((u'_{h},k'_{h}),\dots,(u'_1,k'_1)|$ \quad 
	if  $(u_i,k_i)\leq (u'_i,k'_i)$, for $1\leq i\leq h$;
	\item
	$ |(v_1,l_1),\dots,(v_{h},l_{h}))\leq |(v'_1,l'_1),\dots,(v'_{h},l'_{h}))$ \quad 
	if $(v_i,l_i)\leq (v'_i,l'_i)$, for $1\leq i\leq h$;
	\end{enumerate}
The following statements are immediate from the definition of the partial order.
\begin{enumerate}
	\item 
	
	For any $E, E'\in \cE$ with $h'=sz(E')$, if $E\leq E'$, then $E(h')(L)\leq E'(L)$ and $E(h')(R)\leq E'(R)$;
	\item 
	For $E_1, E_2\in \cE^L_h$, if $E_1\leq E_2$, then for any $E_3\in \cE^R_h$,
	$F(E_1,E_3)\leq F(E_2, E_3)$.
	\item	For $E_1, E_2\in \cE^R_h$, if $E_1\leq E_2$, then for any $E_3\in \cE^L_h$,
	$F(E_3,E_1)\leq F(E_3, E_2)$.
	
\end{enumerate}
\begin{lemma}\label{lemma:compareEE}
	If $E\leq E'$, then $||E||\prec ||E'||$.
\end{lemma}
\begin{proof}
	If $E\in \cE_h^L$ and $E'\in \cE_h^R\cup \cE_{h-1}$ or $E\in \cE_h^R$ and $E'\in \cE_{h=1}$, then $||E||\prec ||E'||$.
	Otherwise,  $E, E'\in \cE^L_h$, $E, E'\in \cE^R_h$ or $E, E'\in \cE_{h-1}$, so if $wt(E)<wt(E')$, then $||E||\prec ||E'||$.
	
	If $E, E'\in \cE^L_h$ and $wt(E)=wt(E')$,  we must have $k_i=k_i'$, so $u_i\leq u'_i$.
If $E, E'\in \cE^R_h$ and $wt(E)=wt(E')$,  we must have $l_j=l_j'$, so $v_j\leq v'_j$.
If $E, E'\in \cE_{h-1}$ and $sz(E)=sz(E')$, we must have $k_i=k_i'$ and $l_j=l_j'$. Then $u_i\leq u'_i$ and $v_j\leq v'_j$, and we have $||E||\prec ||E'||$.
\end{proof}
\subsection{Straightening relations}
\begin{lemma}
	 For any $0\leq l<i$,
	\begin{equation}\label{eqn:basicrelation+4}
	\sum_{\sigma} \frac{\sign(\sigma)}{h!i!} 
	\bpartial^{l} (u'_{h'},\dots,u'_{i+1},u_{\sigma(h+i)},\dots,u_{\sigma(h+1)}|	v_1,\dots v_{h'})
\,\,	\bpartial^{0} (u_{\sigma(h)},\dots,u_{\sigma(1)}|
	=0;
	\end{equation}
\begin{equation}\label{eqn:basicrelation+5}
\sum_{\sigma} \frac{\sign(\sigma)}{h!i!} 
\bpartial^{l} (u_{h'},\dots,u_1|v'_{h'},\dots,v'_{i+1},v_{\sigma(h+i)},\dots,v_{\sigma(h+1)})\,\,	\bpartial^{0} |v_{\sigma(h)},\dots,v_{\sigma(1)})
=0.
\end{equation}
The summations are over all permutations $\sigma$ of $1,2,\dots, h+i$
\end{lemma}
\begin{proof}

	By Proposition \ref{prop:detexpression}, $$\bpartial^{l} (u'_{h'},\dots,u'_{i+1},u_{\sigma(h+i)},\dots,u_{\sigma(h+1)}|	v_1,\dots v_{h'})=\sum_{l=1}^i\sum_{j=1}^h x^{(0)}_{u_{\sigma(h+l)v_j}}f_{lj}.$$
	By Equation (\ref{eqn:basicrelation4}), $\sum_\sigma x^{(0)}_{u_{\sigma(h+l)v_j}}f_{lj} \bpartial^{0} (u_{\sigma(h)},\dots,u_{\sigma(1)}|=0$. So Equation (\ref{eqn:basicrelation+4}) is satisfied. The proof of Equation \eqref{eqn:basicrelation+5} is similar. \end{proof}
	
\begin{lemma}\label{lemma:straighten}
	The elements of $\cJ_h^+$ satisfy the following relations:
\begin{enumerate}
	\item
		\begin{equation}\label{eqn:relationLR+}
	\bpartial^n(u_h,\dots,u_2,u_1|v_1,v_2,\dots,v_h)=\sum_{k=0}^n\bpartial^k(u_h,\dots,u_2,u_1| \,\,\,\bpartial^{n-k}|v_1,v_2,\dots,v_h).
	\end{equation}
	\item \label{enum:relation4+}
	
	For $i_1,i_2,k_0,m,h,h'\in \mathbb Z_{\geq 0}$ with $k_0\leq m$ and $h\geq h'$, let $l_0=i_1+i_2-h-1$. For any integers $a_{k_0+l}$, $0\leq l\leq l_0$, there are integers $a_k$, $0\leq k<k_0$ or $k_0+l_0<l\leq m$, such that
	$$\sum_{k=0}^m a_k\sum_{\sigma} \frac{\sign(\sigma)}{i_1!i_2!}  \left(
	\begin{array}{cc}
	\bpartial^{m-k} (u_h,\dots,u_{i_1+1},\sigma(u_{i_1}),\dots,\sigma(u_1)|& \\
	\bpartial^{k} (u'_{h'},\dots,u'_{i_2+1},\sigma(u'_{i_2}),\dots,\sigma(u'_1)|&v'_1, v'_2,\dots,v'_{h'})
	\end{array}
	\right)
	=0;$$
	Here the second summation is over all permutations $\sigma$ of $u_{i_1},\dots,u_1,u'_{i_2},\dots,u'_1$.
	For simplicity, we write the above equation in the form:
	\begin{equation}\label{eqn:relation1+}
	\sum \epsilon a_{k} \left(
	\begin{array}{cc}
	\bpartial^{m-k} (u_h,\dots,u_{i_1+1},\underline{u_{i_1},\dots,u_1}|& \\
	\bpartial^{k} (u'_{h'},\dots,u'_{i_2+1},\underline{u'_{i_2},\dots,u'_1}|&v'_1, v'_2,\dots,v'_{h'})
	\end{array}
	\right)
	=0.   
	\end{equation}
	Similarly, we have the following relations:
	\begin{eqnarray}
	\label{eqn:relation2+}
	\sum \epsilon a_{k} \left(
	\begin{array}{ccc}
	\bpartial^{m-k} &&|\underline{v_1,\dots,v_{i_1}},v_{i_1+1},\dots,v_h)\\
	\bpartial^{k} (&u'_{h'},\dots,u'_2,u'_1&|\underline{v'_1,\dots, v'_{i_2}},v'_{i_2+1},\dots,v'_{h'})
	\end{array}
	\right)
	=0;
	\\
	\label{eqn:relation3+}
	\sum \epsilon a_{k} \left(
	\begin{array}{c}
	\bpartial^{m-k} (u_h,\dots,u_{i_1+1},\underline{u_{i_1},\dots,u_1}| \\
	\bpartial^{k} (u'_{h},\dots,u'_{i_2+1},\underline{u'_{i_2},\dots,u'_1}|
	\end{array}
	\right)
	=0;\\
	\label{eqn:relation4+}
	\sum \epsilon a_{k} \left(
	\begin{array}{c}
	\bpartial^{m-k} |\underline{v_1,\dots,v_{i_1}},v_{i_1+1},\dots,v_h)\\
	\bpartial^{k} |\underline{v'_1,\dots, v'_{i_2}},v'_{i_2+1},\dots,v'_{h})
	\end{array}
	\right)
	=0.
	\end{eqnarray}

	\item 
	For $i_1,i_2,j_1,j_2,r,h',k_0,m\in \mathbb Z_{\geq 0}$ with $r\geq h'$, $i_1,j_1\leq r$, $i_2,j_2\leq h'$ and $k_0\leq m$, let $l_0=i_1+i_2+j_1+j_2-2r-1$.  Given any integers  $a_{k}$, $k_0\leq k\leq k_0+l_0$, there are integers $a_k$, $0\leq k<k_0$ or $k_0+l_0<k\leq m$, such that 
	\begin{eqnarray}\label{eqn: relation4'}
	\sum_{k=0}^ma_k\sum_{\sigma,\sigma'} \frac{1}{i_1!i_2!j_1!j_2!}\sign(\sigma) \sign(\sigma')\hspace{5cm}&\\  \left(
	\begin{array}{ccc}
	\bpartial^{m-k} (u_r,\dots,u_{i_1+1},\sigma(u_{i_1}),\dots,\sigma(u_1)&|&\sigma'(v_1),\dots, \sigma'(v_{j_1}),v_{j_1+1},\dots,v_r)\\
	\bpartial^{k} (u'_{h'},\dots,u'_{i_2+1},\sigma(u'_{i_2}),\dots,\sigma(u'_1)&|&\sigma'(v'_1),\dots, \sigma'(v'_{j_2}),v'_{j_2+1},\dots,v'_{h'})
	\end{array}
	\right)&
	\in \R[r+1].\nonumber
	\end{eqnarray}
	Here the second summation is over all pairs of permutations $\sigma$ of $u_{i_1},\dots,u_1,u'_{i_2},\dots,u'_1$ and permutations $\sigma'$  of $v_{i_1},\dots,v_1,v'_{i_2},\dots,v'_1$, $\sign(\sigma)$ and $\sign(\sigma')$ are the signs of the permutations, and $\R[r+1]$ is the ideal of $\R$ generated by all $J\in\cJ$ with $sz(J)=r+1$.

\end{enumerate}
\end{lemma}

\begin{remark}
		Since the summation is over all permutations of the underlined elements,  each monomial appear in Equations \eqref{eqn:relation1+}--\eqref{eqn:relation4+} will appear $i_1!i_2!$ times, so the coefficient of each monomial will be $\pm a_k$.
	Since the second summation in Equation (\ref{eqn: relation4'}) is over all pairs of permutations $\sigma$ and $\sigma'$, each monomial in the equation will appear $i_1!i_2!j_1!j_2!$ times, so the coefficient of each monomial will be $\pm a_k$.	
\end{remark}

\begin{proof}
	The statement (1) is directly induced from Equation (\ref{eqn:relationLR}) and Proposition \ref{prop:dab}.
	
	Let $\cF_l(u_h,\dots,u_{i_1+1}; u'_{h'},\dots,u'_{i_2+1})$
	$$
	=
	\sum_{\sigma} \frac{\sign(\sigma)}{i_1!i_2!}  \left(
	\begin{array}{cc}
	\bpartial^{(0)} (u_h,\dots,u_{i_1+1},\sigma(u_{i_1}),\dots,\sigma(u_1)|& \\
	\bpartial^{l} (u'_{h'},\dots,u'_{i_2+1},\sigma(u'_{i_2}),\dots,\sigma(u'_1)|&v'_1, v'_2,\dots,v'_{h'})
	\end{array}
	\right).
	$$
	We have
	$\cF_l(u_h,\dots,u_{i_1}; u'_{h'},\dots,u'_{i_2+1})$
	$$=\cF_l(u_h,\dots,u_{i_1+1};  u'_{h'},\dots,u'_{i_2+1})\pm\cF_l(u_h,\dots,u_{i_1+1}; u'_{h'},\dots,u'_{i_2+1},u_{i_1}).$$
	If $i_1=h$ and $0\leq l\leq l_0$, by Equation (\ref{eqn:basicrelation+4}), $\cF_l(\,\,; u'_{h'},\dots,u'_{i_2+1})=0$.
	By induction on $h-i_1$, we can see that $\cF_l(u_h,\dots,u_{i_1+1}; u'_{h'},\dots,u'_{i_2+1})=0.$
	Thus 
	$$0=\bpartial^{m-l}\cF_l(u_h,\dots,u_{i_1+1}; u'_{h'},\dots,u'_{i_2+1})$$
	$$=\sum_{k=0}^m C_k^l\sum_{\sigma} \frac{\sign(\sigma)}{i_1!i_2!}  \left(
	\begin{array}{cc}
	\bpartial^{m-k} (u_h,\dots,u_{i_1+1},\sigma(u_{i_1}),\dots,\sigma(u_1)|& \\
	\bpartial^{k} (u'_{h'},\dots,u'_{i_2+1},\sigma(u'_{i_2}),\dots,\sigma(u'_1)|&v'_1, v'_2,\dots,v'_{h'})
	\end{array}
	\right) $$
		Now the $(l_0+1)\times (l_0+1)$ integer matrix with entries $c_{ji}=C_{k_0+j}^i$, $0\leq i,j\leq l_0$ is invertible since the determinant of this matrix is $\pm1$. Let $b_{ij} \in \mathbb{Z}$ be the entries of the inverse matrix. Let $a_{k}= \sum_{l=0}^{l_0}\sum_{j=0}^{l_0}C_{k}^l b_{l,j}a_{k_0+j}$. 
	So 
	$$0=\sum_{l=0}^{l_0}\sum_{j=0}^{l_0} b_{l,j}a_{k_0+j}\bpartial^{m-l}\cF_l(u_h,\dots,u_{i_1+1}; u'_{h'},\dots,u'_{i_2+1})=\text{right hand side of Equation (\ref{eqn:relation1+})}
	$$
	
	The proofs for Equations (\ref{eqn:relation2+}, \ref{eqn:relation3+}, \ref{eqn:relation4+})
are similar. Equation (\ref{eqn: relation4'}) is Lemma 2.4 of \cite{LS1}.
\end{proof}

\subsection{Standard monomials} Now we give a definition of the standard monomials of $\cJ_h^+$. Here $\cE(J)$ for $J\in \tilde \cJ_h^+$, is given by \eqref{def:E(J)}.
\begin{defn}
	An ordered product $
	E^L_{1}E^L_{2}\cdots E^L_{m}E^R_1\cdots E^R_n E_1\cdots E_t$ of elements of $\cE_h^+$ is said to be standard
	if 
	\begin{enumerate}
		\item
		$E^L_i\in \cE_h^L$,  $E^R_i\in  \cE_h^R$, $E_i\in \cE_{h-1}$;
		\item  for $1\leq a<m$, $E^L_a\leq E^L_{a+1}$ and $E^L_{a+1}$ is the largest element in $$\{E\in\cE(||E^L_{a+1}||) | E^L_a\leq E\};$$
		\item  for $1\leq a<n$, $E^R_a\leq E^R_{a+1}$ and $E^R_{a+1}$ is the largest element in $$\{E\in\cE(||E^R_{a+1}||) | E^R_a\leq E\};$$
		\item  for $1\leq a<t$, $E_a\leq E_{a+1}$ and $E_{a+1}$ is the largest element in $$\{E\in\cE(||E_{a+1}||) | E_a\leq E\};$$
		\item 
		$E^L_{1}$ and $E^R_1$ are the largest elements in $\cE(||E^L_{1}||)$ and $\cE(||E^R_1||)$ respectively.
		$E_1$ is the largest element in
		 $$\{E\in \cE(||E_{1}||)| E^L_m\leq E \text{ if } m\geq 1 \text{ and } E^n_l\leq E \text{ if } n\geq 1\}.$$
	
	\end{enumerate}
	An ordered product $J_1J_2\cdots J_k$ of elements of $\cJ_h^+$ is said to be standard
	if there is a standard ordered product $E_{1}E_{2}\cdots E_{k}$ such that $E_i\in \cE_h^+(J_i)$.
\end{defn}
\begin{remark}\label{rem:LRcompare}
	Let $E_0^L=((h,0),\dots, (2,0),(1,0)|$ and $E^R_0=|(1,0),(2,0),\dots,(h,0)$. Let $E=F(E_m^L,E_n^R)$, then the condition " $E^L_m\leq E_1$ if $m\geq 1$ and $E^n_l\leq E_1$ if $n\geq 1$", is equivalent to "$E\leq E_1$".
\end{remark}
Let $\cS\cM(\cJ_h^+)\subset \cM(\cJ_h^+)$ be the set of standard monomials of $\cJ_h^+$. 
Let $\cS\cM(\cE_h^+)\subset \cM(\cE_h^+)$ be the set of standard monomials of $\cE_h^+$.

By definition, the standard monomial $E^L_{1}E^L_{2}\cdots E^L_{m}E^R_1\cdots E^R_n E_1\cdots E_{k-m-n}\in \cS\cM(\cE_h^+)$ corresponding to  the standard monomial $J_1J_2\cdots J_k\in \cS\cM(\cJ_h^+)$ is unique and  $E^L_1$ (if $m=0$, left part of $E_1$) has the form $((u_{h'}, k),(u_{h'-1},0),\dots,(u_1,0)|$ and $E^R_1$ (if $n=0$, the right part of $E_1$) has the form $|(v_1,0),\dots,(v_{h'-1},0),(v_{h'},l))$ (if $n=0$, $l=0$) with $u_i<u_{i+1}$ and $v_i<v_{i+1}$. So the map
$$\pi_h^+: \cS\cM(\cE_h^+) \to \cS\cM(\cJ^+_h),\quad E_1E_2\cdots E_m\mapsto ||E_1|| ||E_2|| \cdots ||E_m||$$
is a bijection.

The following lemma will be proved later in Section \ref{sec:base}.
\begin{lemma}\label{lem:base20}
	If $J = J_1\cdots J_b\in\cM(\cJ_h^+)$,  is not standard, $J$ can be written as a linear combination of $\cM(\cJ_h^+)$ preceding $J_1\cdots J_{b-1}$ with integer coefficients.
\end{lemma}
Since $\R_h^+$ is generated by $\cJ_h^+$,
by the above lemma and a standard procedure (for example, proof of theorem 4.3.3.1 in \cite{LR}), we immediately have
\begin{lemma}\label{lem:base21}
	Any element of $\R^+_h$ can be written as a linear combination of standard ordered products of elements of $\cJ^+_h$ with integer coefficients.
\end{lemma}
\begin{proof}
If the lemma is not true,
	there must be a smallest element $J\in \cM(\cJ_h^+)$, which cannot be written as a linear combination of  elements of $\cS\cM(\cJ_h^+)$ with integer coefficients. So $J$ is not standard. By Lemma \ref{lem:base20},
	$J=\sum_{\alpha}\pm  J_{\alpha}$ with $J_{\alpha}\in \cM(\cJ_h^+)$ and $J_{\alpha}\prec J$. By assumption, $J_{\alpha}$ can be written as a linear combination of elements of $\cS\cM(\cJ_h^+)$ with integer coefficients. Thus
	$J$ can be written as a linear combination of elements of $\cS\cM(\cJ_h^+)$ with integer coefficients, which is a contradiction.
\end{proof}

\section{A canonical basis} \label{sec:canonicalbasis} 
In this section, we will prove that the set of standard monomials $\cS\cM(\cJ_h^+)$ is linearly independent. Combined with Lemma \ref{lem:base21}, this shows that $\cS\cM(\cJ_h^+)$ is a $\mathbb{Z}$-basis of $\R^+_h$. The idea of the proof is to construct a ring homomorphism $Q^+_h: \R^+_h \rightarrow \B$, where $\B$ is a certain polynomial ring over $\mathbb{Z}$. By representing the monomials of $\B$ in terms of double tableaux and giving them a total ordering, we will see that for each standard monomial in $\cS\cM(\cJ_h^+)$, the leading term of its image in $\B$ is unique.

Let $$\cS_h=\{a^{(k)}_{il},b^{(k)}_{jl}|\ 1\leq i\leq p, \ 1\leq j\leq q, \ 1\leq l\leq h, \ k\in \Zplus \},$$ and let
\begin{equation} \label{def:ringB} \B = \mathbb{Z}[\cS_h]. \end{equation} Note that for a field $K$, if $W = K^{\oplus h}$ and $V = W^{\oplus p} \bigoplus {W^*}^{\oplus q}$, the affine coordinate ring $K[J_{\infty}(V)]$ is just $\B \otimes_{\mathbb{Z}}K$. There is a derivation $\partial$ on $\B$ defined on generators by
 $\partial a^{(k)}_{ij}=(k+1)a^{(k+1)}_{ij}$ and $\partial b^{(k)}_{ij}=(k+1)b^{(k+1)}_{ij}$. As above, $\bpartial^k = \frac{1}{k!}\partial^k: \B\to \B$ is well-defined.
 
We have a homomorphism of rings
$\tilde Q_h^+:\R[\cJ_h^L\cup\cJ_h^R]\to \B$  given by 
\begin{equation}\label{eqn:eleRL} \tilde Q_h^+(x^{(k)}_{ij})=\bpartial^kX^{(0)}_{ij},
\end{equation}
\begin{equation}\label{eqn:eleJL}
\tilde Q_h^+(\bpartial^k(u_h,\dots,u_1|)=
\bpartial^k Y^{(0)}_{u_1,\dots,u_h},
\end{equation}
\begin{equation}\label{eqn:eleJR}
\tilde Q_h^+(\bpartial^k|v_1,\dots,v_h))=
\bpartial^k Z^{(0)}_{v_1,\dots,v_h},
\end{equation}
where $X^{(0)}_{ij}$, $Y^{(0)}_{u_1,\dots,u_h}$, and $Z^{(0)}_{v_1,\dots,v_h}$ are given by \eqref{eqn:relation0}.
\begin{lemma}
	\label{lem:hom+}
	$\partial \tilde Q_h^+=\tilde Q_h^+\partial$ and
	$\cI^+_h \subset \ker(\tilde Q_h^+)$.
\end{lemma}
\begin{proof} The first statement is obvious from the definition of $\tilde Q_h^+$. For the second statement, we only need to show $\tilde Q_h^+$ maps the right hand sides of Equations (\ref{eqn:relationLR}-\ref{eqn:basicrelation5}) to zero when $n=0$.
This is obvious for Equations (\ref{eqn:relationLR}-\ref{eqn:baiscrelation3}). The proof for Equations (\ref{eqn:basicrelation4}) and (\ref{eqn:basicrelation5}) is similar to the proof of Equation (\ref{eqn:basicrelation+4}), since 
$$\sum_{i=0}^h(-1)^i\tilde Q_h^+(\bpartial^0(u_h,\dots,u_{i+1},u_{i-1},\dots,u_0)) a_{u_i,j}=0.$$	
\end{proof}
By Lemma \ref{lem:hom+}, $\tilde Q_h^+$ induces a ring homomorphism
\begin{equation} \label{def:mapqh} Q_h^+:\R_h^+\to \B, \end{equation} and $\partial Q_h^+=Q_h^+\partial$.

\subsection{Double tableaux}
We use double tableaux to represent monomials of $\B$. Let $\cT_h^+$ be the set of the following tableaux:
\begin{equation}\label{eqn:tableT+}
\left|\begin{array}{ccc}
y_{1,h},\cdots,y_{1,2}, y_{1,1} &| & \\
\vdots & | &\\
y_{m,h},\cdots,y_{m,2}, y_{m,1} & | &\\
& | & z_{1,1},z_{1,2},\cdots, z_{1,h}\\
& | & \vdots\\
& | & z_{n,1},z_{n,2},\cdots, z_{n,h}\\
y'_{1,h},\cdots,y'_{1,2}, y'_{1,1} & | & z'_{1,1},z'_{1,2},\cdots, z'_{1,h} \\
\vdots & \vdots & \vdots \\
y'_{t,h},\cdots,y'_{t,2} y'_{t,1} & | & z'_{t,1}, z'_{t,2}\cdots, z'_{t,h} \\
\end{array}\right|.
\end{equation}
Here $y_{s,h}$ are some $a^{(k)}_{ih}$ and $z_{s,h}$ are some $b^{(k)}_{jh}$;  $y'_{s,h}$ and $z'_{s,h}$ are $*$, $b^{(k)}_{jh}$; for $1\leq l\leq h-1$,
$y_{s,l}$, $y'_{s,l}$ are some $a^{(k)}_{il}$ or  $*$ and $z_{s,l}$, $z'_{s,l}$ are some $b^{(k)}_{jl}$ or $*$; every row of the tableau has elements in $\cS_h$; and
$$y_{s,j}\leq y_{s+1,j}\leq y'_{s',j}\leq y'_{s'+1,j}, \quad z_{s,j}\leq z_{s+1,j}\leq  z'_{s',j}\leq z'_{s'+1,j}.$$
We use the tableau (\ref{eqn:tableT+}) to represent a monomial in $\B$, which is the product of  ${a^{(k)}_{ij}}'s$ and ${b^{(k)}_{ij}}'s$ in the tableau.
It is easy to see that the representation is a one-to-one correspondence between $\cT_h^+$ and the set of monomials of $\B$.
We associate to the tableau (\ref{eqn:tableT+}) the word:
$$ 
y_{1,h}y_{1,h-1}\cdots y_{1,1}\cdots y_{m,2} y_{m,1} z_{1,h}\cdot z_{1,1}\cdots z_{n,1} y'_{1,h-1}\cdots  y'_{1,1} z'_{1,h-1}\cdot z'_{1,1}\cdots z'_{t,h-1}. 
$$
and order these words lexicographically.
For a polynomial $f\in \B$, let $Ld^+(f)$ be its leading monomial in $f$ under the order we defined on $\cT_h^+$.

For $E_1\cdots E_k\in\cS\cM(\cE)$
with \begin{eqnarray*}
E_i&=((u^i_{h_1},k^i_{h_1}),\dots,(u^i_2,k^i_2),(u^i_1,k^i_1)|\in \cE^L_h,\quad \quad\quad\quad\quad \quad\quad\quad\quad\quad &1\leq i\leq m;\\ E_{i+m}&=|(v^i_1,l^i_1),(v^i_2,l^i_2),\dots,(v^i_{h_1},l^i_{h_1}))\in \cE^R_h,\quad \quad\quad\quad\quad \quad\quad\quad\quad\quad &1\leq i\leq n;\\
E_{i+m+n}&=((u^i_{h_1},k^i_{h_1}),\dots,(u^i_2,k^i_2),(u^i_1,k^i_1)|(v^i_1,l^i_1),(v^i_2,l^i_2),\dots,(v^i_{h_1},l^i_{h_1}))\in \cE^L_h, \quad& 1\leq i\leq t,
\end{eqnarray*}
and $k=m+n+t$,
we use the following double tableau to represent it.
\begin{equation}\label{eqn:table+}\left(
\begin{array}{ccc}
(u^1_{h},k^1_{h}),\cdots,(u^1_2,k^1_2),(u^1_1,k^1_1)&|&\\
\vdots &|& \\
(u^m_{h},k^m_{h}),\cdots,(u^m_2,k^m_2),(u^m_1,k^m_1)&|&\\
----------------&|&--------------\\
&|&(v^1_1,l^1_1),(v^1_2,l^1_2),\cdots,(v^1_{h},l^1_{h})\\
&|&\vdots\\
&|&
(v^n_1,l^n_1),(v^n_2,l^n_2),\cdots,(v^n_{h},l^n_{h})\\
----------------&|&--------------\\
(u'^1_{h_1},k'^1_{h_1}),\cdots,(u'^1_2,k'^1_2),(u'^1_1,k'^1_1)&|&(v'^1_1,l'^1_1),(v'^1_2,l'^1_2),\cdots,(v'^1_{h'_1},l'^1_{h_1})\\
(u'^2_{h_2},k'^2_{h_2}),\cdots,(u'^2_2,k'^2_2),(u'^2_1,k'^2_1)&|&(v'^2_1,l'^2_1),(v'^2_2,l'^2_2),\cdots,(v'^2_{h_2},l'^2_{h_2})\\
&\vdots & \\
(u'^t_{h_t},k'^t_{h_t}),\cdots,(u'^t_2,k'^t_2),(u'^t_1,k'^t_1)&|&(v'^t_1,l'^t_1),(v'^t_2,l'^t_2),\cdots,(v'^t_{h_t},l'^t_{h_t})\\
\end{array}
\right).
\end{equation}
Let $T^+:\cS\cM(\cJ_h^+))\to \cT^+$ with
$$T^+(E_1\cdots E_k)=\left(
\begin{array}{ccc}
a_{u^1_{h} h}^{(k^1_{h})}, \,\,\,\cdots,\,\,\, a_{u^1_{2} 2}^{(k^1_{2})},\,\, a_{u^1_{1}1}^{(k^1_{1})} &|&\\
\vdots &|&\\
a_{u^m_{h} h}^{(k^m_{h})},\,\,\, \cdots,\,\,\, a_{u^m_{2}2}^{(k^m_{2})},\,\, a_{u^m_{1}1}^{(k^m_{1})} &|&\\
&|& b_{v^1_{1}1}^{(l^1_{1})},\,\,b_{v^1_{2}2}^{(l^1_{2})},\,\,\,\cdots,\,\,\,b_{v^1_{h} h}^{(l^1_{h})}\\
&|& \vdots\\
&|& b_{v^n_{1}1}^{(l^n_{1})},\,\,b_{v^n_{2}2}^{(l^n_{2})},\,\,\,\cdots,\,\,\,b_{v^n_{h} h}^{(l^n_{h})}\\
*,\cdots, *, a_{u'^1_{h'_1}h'_1}^{(k'^1_{h'_1})},\cdots,a_{u'^1_{1}1}^{(k'^1_{1})}&|&b_{v'^1_{1}1}^{(l'^1_{1})},\cdots,b_{v'^1_{h'_1}h'_1}^{(l'^1_{h'_1})},*,\cdots,*\\
\vdots&|&\vdots \\
*,\cdots, *, a_{u'^t_{h'_t}h'_t}^{(k'^t_{h'_t})},\cdots,a_{u'^t_{1}1}^{('k^t_{1})}&|&b_{v'^t_{1}1}^{(l'^t_{1})},\cdots,b_{v'^t_{h'_t}h'_t}^{(l'^t_{h'_t})},*,\cdots,*
\end{array}
\right).
$$
Then $T^+$ is injective and $T(E_1)\prec T(E_2)$ if $E_1\prec E_2$.
\begin{lemma}\label{lem:leadingterm+}
	Let $J_1\cdots J_k\in \cS\cM(\cJ_h^+)$ and $E_1\cdots E_k\in \cS\cM(\cE_h^+)$ be its associated standard order product of elements of $\cE_h^+$. Assume the double tableau representing $E_1\cdots E_k$ is (\ref{eqn:table+}). Then
	the leading monomial of $Q_h^+(J_1\cdots J_k)$ is represented by the double tableau $T^+(E_1\cdots E_k)$.
	So $$Ld^+\circ Q_h^+=T^+\circ(\pi_h^+)^{-1}: \cS\cM(\cJ_h^+)\to \cT_h^+$$ is injective. The coefficient of the leading monomial of $Q_h^+(J_1\cdots J_m)$ is $\pm 1$.
\end{lemma}
\begin{proof}
	Let $W_r$ be the monomial corresponding to the double tableau $T(E_1\cdots E_r)$.
Let $
	M_r$
	be the monomial corresponding to the double tableau $T(E_r)$.
	Then $W_r=W_{r-1}M_r$.
	We prove the lemma by induction on $r$. If $r=1$, the lemma is obviously true. Assume the lemma is true for $J_1\cdots J_{r-1}$. Then $Ld(Q_h^+(J_1\cdots J_{r-1}))=W_{n-1}$, the monomial corresponds to $T(E_1\cdots E_{r-1})$, and the coefficient of $W_{n-1}$ in $Q_h^+(J_1\cdots J_{r-1})$ is $\pm 1$. Let $h_r=h$ if $r\leq m+n$.
		\begin{eqnarray*}
	Q_h^+(J_r)&=&\sum \pm a_{u^r_{1} s_1}^{(k_1)}a_{u^r_{2} s_2}^{(k_2)}\cdots a_{u^r_{h_r}s_{h_r}}^{(k_{h_r})}, \quad\quad \quad\quad\quad\quad\quad\quad r\leq m;
	\\
			Q_h^+(J_r)&=&\sum \pm b_{v^r_{1} t_1}^{(l_1)}b_{v^r_{2}t_2}^{(l_2)}\cdots b_{v^r_{h_r} t_{h_r}}^{(l_{h_r})}, \quad\quad \quad\quad\quad\quad \quad\quad\quad m<r\leq m+n;
	\\
	Q_h^+(J_r)&=&\sum \pm a_{u^r_{1} s_1}^{(k_1)}a_{u^r_{2} s_2}^{(k_2)}\cdots a_{u^r_{h_r}s_{h_r}}^{(k_{h_r})}b_{v^r_{1} t_1}^{(l_1)}b_{v^r_{2}t_2}^{(l_2)}\cdots b_{v^r_{h_r} t_{h_r}}^{(l_{h_r})}, \quad r>m+n.
	\end{eqnarray*}
If $r<m$, the summation is over all  $k_i\geq 0$ with $\sum k_i=wt(E_r)$, and all $s_i$ with $1\leq s_1, s_2,\dots,s_{h_r}\leq h$ and they are different from each other. If $m<r\leq m+n$,  the summation is over all $l_i\geq 0$ with $\sum l_i=wt(E_r)$, and all $t_i$ with $1\leq t_1,t_2,\dots,t_{h_r}\leq h$ and they are different from each other. If $r>m+n$,
  the summation is over all $l_i, k_i\geq 0$ with $\sum (l_i+k_i)=wt(E_r)$, all $s_i$ with $1\leq s_1, s_2,\dots,s_{h_r}\leq h$ and they are different from each other, and all $t_1,\dots, t_{h_r}$, which are permutations of $s_1,\dots, s_{h_r}$.
	$M_r$ is one of the monomials in $Q_h^+(J_r)$ with coefficient $\pm 1$.
	All of the monomials in $Q_h^+(J_1\cdots J_{r-1})$  except $W_{n-1}$ are less than
	$W_{n-1}$, so any monomial in $Q_h^+(J_1\cdots J_{r-1})$ except $W_{n-1}$ times any monomial in $Q_h^+(J_r)$  is less than $W_{r-1}$.  The coefficient of $W_r$ in $Q^+_h(J_1\cdots J_r)$
	is $\pm 1$ since $W_{r-1}\prec W_r$, and the coefficients of $W_{r-1}$ in $Q^+_h(J_1\cdots J_{r-1})$ and $M_r$ in $Q^+_h(J_r)$ are $\pm 1$.   Now $$W_{r-1}\prec W_r\prec Ld(Q_h^+(J_1\cdots J_{r})).$$
	The leading monomial $Ld(Q_h^+(J_1\cdots J_{r}))$ must have the form
	$$W=W_{r-1} a_{u^r_{1} s_1}^{(k_1)}a_{u^r_{2} s_2}^{(k_2)}\cdots a_{u^r_{h_r}s_{h_r}}^{(k_{h_r})}, r\leq m, \quad W=W_{r-1} b_{v^r_{1} t_1}^{(l_1)}b_{v^r_{2}t_2}^{(l_2)}\cdots b_{v^r_{h_r} t_{h_r}}^{(l_{h_r})}, m<r\leq m+n$$
	$$ W=W_{r-1}a_{u^r_{1} s_1}^{(k_1)}a_{u^r_{2} s_2}^{(k_2)}\cdots a_{u^r_{h_r}s_{h_r}}^{(k_{h_r})}b_{v^r_{1} t_1}^{(l_1)}b_{v^r_{2}t_2}^{(l_2)}\cdots b_{v^r_{h_r} t_{h_r}}^{(l_{h_r})}, \quad r>m+n.$$
	If some $s_i$ or $t_i$ is greater than $h_{r-1}$, then $W\prec W_{n-1}$.
	If there is some $h_{r-1}\geq s_i>h_r$, there is  $1\leq j\leq h_r$, with $j\notin\{s_1,\dots s_{h_r}\}$; if we replace $s_i$ by $j$ in $W$, we get a larger monomial in $Q_h^+(J_1\cdots J_{r})$. So we can assume $s_1,\dots, s_{h_r}$ is a permutation of $1,2,\dots, h_r$. Similarly, we can assume  $t_1,\dots, t_{h_r}$ is a permutation of $1,2,\dots, h_r$.
	We must have $a_{u^r_{i}s_i}^{(k_{i})}\geq a_{u^{r-1}_{s_i}s_i}^{(k^{r-1}_{s_i})}$ if $1<r\leq m$ or $r>m+n+1$, $b_{v^r_{i}t_i}^{(k_{i})}\geq b_{v^{r-1}_{t_i}t_i}^{(k^{r-1}_{t_i})}$ if $m+1<r$, and $a_{u^r_{i}s_i}^{(k_{i})}\geq a_{u^{m}_{i}s_i}^{(k^{m}_{s_i})}$ if $r=m+n+1$ and $m>0$, otherwise $W\prec W_{r-1}$. Such monomials in 
	$Q_h^+(J_r)$ are in one-to-one correspondence with $E'_r\in\cE(J_r)$ which satisfy the property\\
	(P):\quad \textsl{  $E_{r-1}\leq  E'_r$ if $r\neq 1,m+1,m+n+1$ and $E_m\leq E'_r$ and $E_{r-1}\leq E'_r$ if $r=m+n+1$ and $m>0$. } \\
	$E_{r}$ is the largest element in $\cE(J_r)$ satisfying the property (P), since $E$ is standard, so $W_r$ is the leading term of $Q_h^+(J_1\cdots J_r)$. 
\end{proof}

\begin{thm}\label{thm:standmonomial1+} $Q_h^+:\R_h^+\to \B$ is injective, so we may identify $\R^+_h$ with the image $\text{Im}(Q_h^+)$, which is the subring of $\B$ generated by $\bpartial^k X^{(0)}_{ij}$,  $\bpartial^k Y^{(0)}_{u_1,\dots, u_h}$, and $\bpartial^k Z^{(0)}_{v_1,\dots, v_h}$. In particular, the set of standard monomials $\cS\cM(\cJ_h^+)$ is a $\mathbb Z$-basis of $\R_h^+$, so $Q_h^+(\cS\cM(\cJ_h^+))$ is a $\mathbb Z$-basis of the subring  $\text{Im}(Q_h^+)$.
\end{thm}
\begin{proof} By Lemma \ref{lem:leadingterm+}, $Ld^+(Q_h^+(\cS\cM(\cJ_h^+)))$ is linearly independent, so $\cS\cM(\cJ_h^+)$ is linearly independent. By Lemma \ref{lem:base21}, $\cS\cM(\cJ_h^+)$ spans $\R_h^+$.  So $\cS\cM(\cJ_h^+)$ is a $\mathbb Z$-basis of $\R_h^+$ and $Q_h^+$ is injective. 
\end{proof}
Since $Q_h^+$ is injective and $\B$ is an integral domain, we obtain
\begin{cor} $\R^+_h$ is an integral domain.
\end{cor}

Let $K$ be an algebraically closed field, and consider the induced map
\begin{equation} \label{def:mapqpk} Q^{+K}_h=Q^+_h\otimes Id: \R_h^+\otimes_{\mathbb{Z}} K\to \B\otimes_{\mathbb{Z}} K.
\end{equation}
By Theorem \ref{thm:standmonomial1+}, we immediately have

\begin{cor} \label{thm:injectiveGL} $\cS\cM(\cJ_h^+)$ is a $K$-basis of $\R_h^+\otimes_{\mathbb{Z}} K$. Therefore ${Q^{+K}_h}$ is injective, so $\R_h^+\otimes_{\mathbb{Z}} K$ is integral and $Q^+_h(\cS\cM(\cJ_h^+))$ is a $K$-basis of the subring $\text{Im}(Q_h^{+K})$.
\end{cor}

\section{Application} \label{sec:application}
In this section, we use the standard monomial basis we have constructed to prove Theorem \ref{thm:JSLinvariant+}. 

\subsection{Arc spaces}
Suppose that $X$ is a scheme of finite type over $K$. Recall that its arc space $J_\infty(X)$ is determined by its functor of points: or every $K$-algebra $A$, we have a bijection
$$\Hom(\Spec A, J_\infty(X))\cong\Hom(\Spec A[[t]], X).$$
If $f: X\to Y$ is a morphism of schemes, we get a morphism of schemes $f_{\infty}:J_{\infty}(X)\to J_{\infty}(Y)$. If $f$ is a closed immersion, then $f_{\infty}$ is also a closed immersion.

If $X=\Spec K[x_1,\dots,x_n]$, then $J_{\infty}(Y)=\Spec K[x^{(k)}_i|1\leq i\leq n, k\in \Zplus]$. The identification is made as follows: for a $K$-algebra $A$, a morphism $\phi: K[x_1,\dots, x_n]\to A[[t]]$ determined by $\phi(x_i)=\sum_{k=0}^\infty a_i^{(k)}t^k$ corresponds to a morphism
$K[x_i^{(k)}]\to A$ determined by $x_i^{(k)}\to a_i^{(k)}$. Note that $K[x_1,\dots,x_n]$ can naturally be identified with the subalgebra $K[x^{(0)}_1,\dots,x^{(0)}_n]$ of $K[x_i^{(k)}]$, and from now on we use $x_i^{(0)}$ instead of $x_i$.

The polynomial ring $K[x^{(k)}_i]$ has a derivation $\partial$ defined on generators by 
\begin{equation} \label{def:partial} \partial x_i^{(k)}=(k+1)x_i^{(k+1)}.\end{equation}
It will be more convenient to work with the normalized derivation $\frac 1{n!}\partial^n$, but this is a priori not well-defined on $K[x_i^{(k)}]$ if $\text{char}\ K$ is positive. However, $\partial$ is well defined on $\mathbb Z[x_i^{(k)}]$, and $\bpartial^n = \frac 1{n!}\partial^n$ maps $\mathbb Z[x_i^{(k)}]$ to itself. Therefore for any $K$, there is an induced $K$-linear map 
\begin{equation} \label{def:bpartial} \bpartial^n:K[x_i^{(k)}]\to K[x_i^{(k)}],\qquad \bpartial^n x_i^{(k)}=C_{k+n}^n  x_i^{(k+n)},\end{equation} obtained by tensoring with $K$. 
Here $C_{k+n}^n$ is given by\eqref{def:ckn}.

If $X$ is the affine scheme with affine coordinate ring $K[X] = K[x^{(0)}_1,\dots, x^{(0)}_n]/(f_1,\dots, f_r)$, then $J_\infty(X)$ is the affine scheme with affine coordinate ring
$$K[J_{\infty}(X) ]= K[x_i^{(k)}|\ i =1,\dots, n,\ k \in \Zplus]/ ( \bpartial^l f_j| \ j = 1,\dots, r,\ l \geq 0 ).$$ For every $f\in K[x^{(0)}_1,\dots, x^{(0)}_n]$, we have 
$$\phi(f)=\sum_{k=0}^\infty (\bpartial^kf)(a_1^{(0)},\cdots, a_n^{(k)}) \,t^k.$$ Therefore $\phi$ induces a morphism $K[x^{(0)}_1,\dots, x^{(0)}_n]/(f_1,\cdots f_r)\to A[[t]]$ if and only if 
$$(\bpartial^kf_i)(a_1^{(0)},\cdots, a_n^{(k)})=0,\ \text{for all} \ i = 1,\dots, r,\ k\geq 0.$$

If $Y$ is the affine scheme with affine coordinate ring $K[Y]=K[y^{(0)}_1,\dots,y^{(0)}_m]/(g_1,\dots,g_s)$, a morphism $P:X\to Y$ induces a ring homomorphism $P^*:K[Y]\to K[X]$.
Then the induced homomorphism of arc spaces $P_\infty:J_\infty(X)\to J_\infty(Y)$ is given by
$$P^*_\infty(y_i^{(k)})=\bpartial^k P^*(y_i^{(0)}).$$
In particular, $P^*_\infty$ commutes with $\bpartial^k$ for all $k\geq 0$.

\subsection{Invariant theory for arc spaces}
Let $G=SL_h(K)$, and let $W=K^{\oplus h}$ be its standard representation. As before, for $p,q \geq 0$, let $V=W^{\oplus p}\bigoplus {W^*}^{\oplus q}$. The group structure $G\times G\to G$
induces the group structure on its arc space
$$J_\infty(G)\times J_\infty(G)\to J_\infty(G).$$ 
Therefore $J_\infty(G)$ is an algebraic group. The action of $G$ on $V$, $G\times V\to V$,
induces an action of $J_\infty(G)$ on
$J_\infty (V)$,
$$J_\infty(G)\times J_\infty(V)\to J_\infty(V).$$
The affine quotient map 
$q:V\to V/\!\!/G$ induces $q_\infty: J_\infty(V)\to J_\infty(V/\!\!/G)$ and 
\begin{equation} \label{def:barqinfty} \bar q_\infty: J_\infty(V)/\!\!/J_\infty(G)
\to J_\infty(V/\!\!/G).\end{equation}

Recall the ring $\R^+_h$ defined in \eqref{def:ringrplush}. Define the larger ring 
\begin{equation} \label{def:ringtilderplush} \tilde \R_h^+ = \R / \tilde \cI_h^+,\end{equation} where $\tilde \cI_h^+ \subset \cI_h^+$ is the ideal generated by relations \eqref{eqn:relation0},  \eqref{eqn:relation1}, and \eqref{eqn:relation2}, and their $n^{\text{th}}$ normalized derivatives. Clearly we have a surjective homomorphism $\tilde \R_h^+ \rightarrow \R^+_h$, and the affine coordinate ring of $J_\infty(V/\!\!/G)$ is
$$K[J_\infty(V/\!\!/G)] = \tilde \R_h^+ \otimes_{\mathbb{Z}} K.$$
Identifying $K[J_{\infty}(V)]$ with $\B \otimes_{\mathbb{Z}} K$, we have the induced map 
\begin{equation} \label{eqn:mapq}  q_\infty^*:K[J_\infty(V/\!\!/G)] \to K[J_{\infty}(V)], \end{equation} whose image $\text{Im}(q_{\infty}^*)$ is $J_{\infty}(G)$-invariant.
Note that map $\tilde Q_h^+: \R^+_h \rightarrow \B$ induces a map $\tilde \R_h^+ \rightarrow \B$, and after tensoring with $K$ we recover \eqref{eqn:mapq}

Identifying $\bpartial^0(h,\dots,2,1|) \in \R^+_h$ with an element of $K[J_\infty(V/\!\!/G)]$, let 
$$\Gamma= q_{\infty}^*(\bpartial^0(h,\dots,2,1|) \in K[J_{\infty}(V)],$$ and let $K[J_{\infty}(V)]_\Gamma$ and $\text{Im}(q_{\infty}^*)_\Gamma$ be the localizations of $K[J_{\infty}(V)]$ and $\text{Im}(q_{\infty}^*)$ at $\Gamma$.
\begin{lemma}\label{lemma:equalSLinvariant}If $p\geq h$,
	$${K[J_{\infty}(V)]_\Gamma}^{J_\infty({SL}_h(K))} =\text{Im}(q_{\infty}^*)_\Gamma.$$ 
\end{lemma}
\begin{proof} Let $G=SL_h(K)$, and let $K[V]_\Gamma$ be the localization of the coordinate ring $K[V]$ at $\Gamma$ and
	$V_\Gamma=\Spec K[V]_\Gamma$. Let $H$ be the subvariety of $V_\Gamma$ given by the ideal generated by $a_{il}-\delta_i^l$ with $1\leq i,l\leq h$ and $i+l> 2$. 
	The composition of the embedding $\iota: H\hookrightarrow V_\Gamma$ and the affine quotient $q^\Gamma: V_\Gamma \to V_\Gamma/\!\!/G=(V/\!\!/G)_\Gamma$ gives the isomorphism $q^{\Gamma}\circ \iota:H \to (V/\!\!/G)_\Gamma$. So we have an isomorphism of arc spaces
	$$q^{\Gamma}_\infty\circ \iota_\infty :J_\infty(H) \to J_\infty((V/\!\!/G)_\Gamma)=J_\infty(V/\!\!/G)_\Gamma.$$ $q^{\Gamma}_\infty$ induces a morphism
	$\bar q^{\Gamma}_\infty: J_\infty(V_\Gamma)/\!\!/J_{\infty}(G) \to J_\infty(V/\!\!/G)_\Gamma$.
	The action of $G$ on $V$ gives a $G$-equivariant isomorphism
	$$G\times H\to V_\Gamma.$$
	So we have a $J_{\infty}(G)$-equivariant isomorphism
	$$J_{\infty}(G)\times J_\infty(H)\to J_\infty(V_\Gamma)=J_\infty(V)_\Gamma.$$
	and an isomorphism of their affine quotients
	$$i: J_\infty(H)=J_{\infty}(G)\times J_\infty(H)/\!\!/G\cong J_\infty(V)/\!\!/J_{\infty}(G).$$
	$\bar q^{\Gamma}_\infty \circ i=q^{\Gamma}_\infty\circ \iota_\infty$ is an isomorphism,
	so $\bar q^{\Gamma}_\infty$ is an isomorphism since $i$ is an isomorphism. So
	${K[J_{\infty}(V)]_\Gamma}^{J_\infty(G)}=\bar q^{\Gamma *}_\infty(K[J_\infty(V/\!\!/G)_\Gamma)=\text{Im}(q_{\infty}^*)_\Gamma$.
\end{proof}
\begin{thm}\label{thm:JSLinvariant} $K[J_{\infty}(V)]^{J_\infty({SL}_h(K))} =\text{Im}(q_{\infty}^*)$. 
\end{thm}
\begin{proof} Let $G=SL_h(K)$.
	If $p\geq h$, regard $K[J_{\infty}(V)]$ and $\text{Im}(q_{\infty}^*)_{\Gamma}$ as subrings of $K[J_{\infty}(V)]_{\Gamma}$.
	By Lemma \ref{lemma:equalSLinvariant}, we have  $${K[J_{\infty}(V)]}^{J_{\infty}(G)}=K[J_{\infty}(V)] \cap \text{Im}(q_{\infty}^*)_{\Gamma}.$$ Now for any $f \in K[J_{\infty}(V)] \cap \text{Im}(q_{\infty}^*)_{\Gamma}$, $f=\frac{g}{\Gamma^n}$ with $\Gamma^nf=g\in \text{Im}(q_{\infty}^*)$. The leading monomial of $g$ is $$Ld(g)=(a^{(0)}_{11}\cdots a^{(0)}_{hh})^n Ld(f)$$
	with coefficient $C_0\neq 0$.
	Since $g\in \text{Im}(q_{\infty}^*)$, there is a standard monomial $J\in\cS\cM(\cJ_h^+)$, with $Ld(Q^+_h(J))=Ld(g)$. Since $J$ has the factor $(h, \dots, 1|^n$,  $q_{\infty}^*(J)$ has the factor $\Gamma^n$. Thus $f-C_0\frac{q_{\infty}^*(J)}{{\Gamma}^n}\in K[J_{\infty}(V)] \cap \text{Im}(q_{\infty}^*)_{\Gamma}$ with a lower leading monomial, and $\frac{q_{\infty}^*(J)}{{\Gamma}^n}\in \text{Im}(q_{\infty}^*)$. By induction on the leading monomial of $f$, $f\in \text{Im}(q_{\infty}^*)$. So $$K[J_{\infty}(V)] \cap\text{Im}(q_{\infty}^*)_{\Gamma}=\text{Im}(q_{\infty}^*),$$ and ${K[J_{\infty}(V)]}^{J_{\infty}(G)} =\text{Im}(q_{\infty}^*)$.
	
	More generally, let $V' = W^{\oplus p +h} \bigoplus {W^*}^{\oplus q}$, where $W = K^{\oplus h}$ as above. Its arc space has affine coordinate ring
	$$K[J_{\infty}(V')] = K[a^{(k)}_{il},b^{(k)}_{jl}|\ 1\leq i\leq p+h,\ 1\leq j\leq q,\ k\in \Zplus ],$$ which contains $K[J_{\infty}(V)]$ as a subalgebra, and has an action of $J_{\infty}(G)$. As shown above, $K[J_{\infty}(V')]^{J_{\infty}(G)}$ is generated by $\bpartial^kX^{(0)}_{ij}=\bpartial^k\sum_l a^{(0)}_{il}b^{(0)}_{jl}$ and $q_{\infty}^*(J)$ with $J=\bpartial^k(i_h,\dots, i_1|$  or $J=\bpartial^k|j_1,\dots, j_h)$.

	Let ${\cI}$ be the ideal of $K[J_{\infty}(V')]$ generated by $a^{(k)}_{il}$ with $i>p$. Then  $$K[J_{\infty}(V')]=K[J_{\infty}(V)] \oplus {\cI}.$$
	Note that $K[J_{\infty}(V')]$ amd ${\cI}$ are $J_{\infty}(G)$-invariant subspace of $K[J_{\infty}(V')]$, and
	$$K[J_{\infty}(V')]^{J_{\infty}(G)}=K[J_{\infty}(V)]^{J_{\infty}(G)}\oplus {\cI}^{J_{\infty}(G)}.$$
	If $i>p$, $\bpartial^k X^{(0)}_{ij}\in {\cI}^{J_{\infty}(G)}$. If $J=\bpartial^k(i_h,\dots, i_1|$ with $i_h>p$, then $q_{\infty}^*(J) \in \cI$. So $$K[J_{\infty}(V)]^{J_{\infty}(G)} \cong K[J_{\infty}(V')]^{J_{\infty}(G)} \slash \cI^{J_{\infty}(G)},$$ is generated by $\bpartial^k X^{(0)}_{ij}$,  $1\leq i\leq p$, $1\leq j\leq q$ and $q_{\infty}^*(J)$, $J\in \cJ_h^L\cup \cJ_h^R$. Therefore $K[J_{\infty}(V)]^{J_{\infty}(G)}=\text{Im}(q_{\infty}^*)$. 
\end{proof}

\begin{proof}[Proof of Theorem \ref{thm:JSLinvariant+}]
 By Corollary \ref{thm:injectiveGL} and  Theorem \ref{thm:JSLinvariant}, 
 $$\text{Im}(q_{\infty}^*) \cong \R_h^+ \otimes_{\mathbb{Z}} K \cong K[J_{\infty}(V)]^{J_\infty({SL}_h(K))}.$$ Then according to the definition of $\R_h^+$, the theorem is true. 
\end{proof}

Now by Theorem \ref{thm:JSLinvariant}, $\text{Im}(q_{\infty}^*)=K[J_\infty(V)/\!\!/J_\infty(G)]$. The map $\bar q_\infty : J_\infty(V)/\!\!/J_\infty(G)
\to J_\infty(V/\!\!/G)$ induces a ring homomorphism of affine coordinate rings 
$$\bar q^*_\infty: K[J_\infty(V/\!\!/G)] \to K[J_\infty(V)/\!\!/J_\infty(G)]$$ of affine coordinate rings given by (\ref{eqn:mapq}). By Theorem \ref{thm:JSLinvariant+}, $\bar q^*$ is surjective and the kernel of $\bar q^*$ is generated by Equation (\ref{eqn:relationinvariant2}) and (\ref{eqn:relationinvariant3}) for $l\neq 0$.

\begin{cor}If $2<h<\max(p,q)-2$,
	$\bar q_\infty : J_\infty(V)/\!\!/J_\infty(G)
	\to J_\infty(V/\!\!/G)$ is not an isomorphism and $J_\infty(V/\!\!/G)$ is not integral. Otherwise, $\bar q_\infty$ is an isomorphism and $J_\infty(V/\!\!/G)$ is integral.
\end{cor}

\begin{proof}
 $\A=K[X^{(k)}_{ij},Y^{(k)}_{u_1,\dots,u_h},Z^{(k)}_{v_1,\dots,v_h}]$ has a $\Zplus\times \Zplus \times \Zplus $ grading given by 
$$\deg X^{(k)}_{ij}=(1,1,k),\quad \deg Y^{(k)}_{u_1,\dots,u_h}=(h,0,k),\quad \deg Z^{(k)}_{v_1,\dots,v_h}=(0,h,k).$$
If $p>h-2$, we consider $\A[2h,0,1]$, the subspace of elements with degree $(2h,0,1)$. It is a linear combination of 
$Y^{(0)}_{u_1,\dots,u_h}Y^{(1)}_{u'_1,\dots,u'_h}$.
Its intersection with the ideal $\A[2h,0,1]\cap \cI_h^+$ is a linear combination of 
\begin{equation}
\label{eqn:relationYY}
\sum_{i}(-1)^i(
Y^{(1)}_{u'_{h},\dots,u'_{2},u_i} Y^{(0)}_{u_{1},\dots,u_{i-1},u_{i+1},\dots,u_{h}}+Y^{(0)}_{u'_{h},\dots,u'_{2},u_i} Y^{(1)}_{u_{1},\dots,u_{i-1},u_{i+1},\dots,u_{h}}).
\end{equation}
Let 
$$f=\sum_{1\leq i<j\leq h+2}(-1)^{i+j}Y^{(0)}_{h+3,\dots,6,i,j}Y^{(1)}_{1,\dots,i-1,i+1,\dots,j-1,j+1,\dots,h+2},$$ then $\bar q^*_\infty(f)=0$ by Equation (\ref{eqn:basicrelation4}). But $f$ is not in $\tilde \cI_h^+$ since in (\ref{eqn:relationYY}), $Y^{(1)}_{h+3,\dots}Y^{(0)}_{\dots}$ and $Y^{(0)}_{h+3,\dots}Y^{(1)}_{\dots}$ appear in pairs with the same sign, while in $f$ only $Y^{(0)}_{h+3,\dots}Y^{(1)}_{\dots}$ appears. So $\bar q^*_\infty$ is not injective and $\bar q_\infty$ is not an isomorphism.

Recall the element $\Gamma= q_{\infty}^*(\bpartial^0(h,\dots,2,1|) \in K[J_{\infty}(V)]$. By abuse of notation we 
also denote the corresponding elements of $K[J_\infty(V/\!\!/G)]$ and $K[J_\infty(V)/\!\!/J_{\infty}(G)]$ by $\Gamma$, and we denote by $K[J_\infty(V/\!\!/G)]_{\Gamma}$ and $K[J_\infty(V)/\!\!/J_{\infty}(G)]_\Gamma$ the corresponding localizations. By Lemma \ref{lemma:equalSLinvariant},
$\bar q^{\Gamma *}_\infty: K[J_\infty(V/\!\!/G)]_{\Gamma}  \to K[J_\infty(V)/\!\!/J_{\infty}G)]_\Gamma$ is an isomorphism, so $f$ is zero in $K[J_\infty(V/\!\!/G)]_\Gamma$. Then there is some $n\in \Zplus$, with $\Gamma^n f\in \cI^+_h$, so $K[J_\infty(V/\!\!/G)]$ is not integral. Similarly, if $q>h-2$, we can show that $\bar q_\infty$ is not an isomorphism and $K[J_\infty(V/\!\!/G)]$ is not integral.

Next, $\bar q_{\infty}^*$ maps the generators of $K[J_\infty(V/\!\!/G)]$ to the generators of $K[J_\infty(V)/\!\!/J_{\infty}(G)]$, and maps the generators of the ideal $ \cI_h^+$ to the relations (\ref{eqn:relationinvariant1}), and relations (\ref{eqn:relationinvariant2}) and (\ref{eqn:relationinvariant3}) for $l=0$.
If $h\leq 2$ or $h\geq \max(p,q)-2$, it is easy to check that relations (\ref{eqn:relationinvariant2}) and (\ref{eqn:relationinvariant3}) are trivial for $2\leq l<i$. Also, these relations are the same for $l=0$ and $l=1$. So $K[J_\infty(V/\!\!/G)]$ and $K[J_\infty(V)/\!\!/J_{\infty}(G)]$ have the same generators and relations, so $\bar q^*_{\infty}$ is an isomorphism. Finally, $K[J_\infty(V/\!\!/G)]$ is integral since $K[J_\infty(V)/\!\!/J_{\infty}(G)]$ is integral.
\end{proof}
\begin{cor}If $\text{char}\ K=0$, the nilradical $\cN$ in $K[J_\infty(V/\!\!/G)]$ is the kernel of $\bar q^*_{\infty}$, which is generated by Equation (\ref{eqn:relationinvariant2}) and (\ref{eqn:relationinvariant3}) for $l\neq 0$.
\end{cor}
\begin{proof}Since $K[J_\infty(V)/\!\!/J_{\infty}(G)]$ is integral, $\cN$ is in the kernel. 
	By Lemma \ref{lemma:equalSLinvariant}, $J_\infty(V_\Gamma)/\!\!/J_\infty(G)
	\to J_\infty(V/\!\!/G)_\Gamma$ is an isomorphism. If there is $f\in K[J_\infty(V/\!\!/G)]$, $\bar q^*_\infty(f)=0$, we must have $(Y^{(0)}_{h,\dots,1})^n f=0$.
	If $\text{ char } K=0$, $K[J_\infty(V/\!\!/G)]$ is irreducible, so $f\in \cN$ since $Y^{(0)}_{h,\dots,1}$ is not in the nilradical.
\end{proof}

\section{Some properties of standard monomials} \label{sec:properties}
In this section, we study the properties of $||\cE(E_{i+1})||$ and $E_{i+1}$ that need to be satisfied to make $E_1E_2\cdots E_n$ a standard monomial. These properties are similar to the properties of standard monomials in \cite{LS1}, and so are their proofs.

Let 
$$E^L=((u_h,k_h),\dots,(u_1,k_1)| \in \cE^L_h,\quad J'^L=\bpartial^{k'}(u'_{h'},\dots,u'_2,u'_1|\in \cJ^L_{h'},$$ $$E^R=|(v_1,l_1),\dots,(v_h,l_h))\in \cE^R_h,\quad J'^R=\bpartial^{l'}|v'_1,v'_2,\dots,v'_{h'})\in\cJ^R_{h'},$$
$$E=((u_h,k_h),\dots,(u_1,k_1)|(v_1,l_1),\dots,(v_h,l_h))\in \cE,\quad 
J'=\bpartial^{n'}(u'_{h'},\dots,u'_1|v'_1,\dots,v'_{h'})\in \cJ.$$
For $h'\leq h$, let
$$E^L(h')=((u_{h'},k_{h'}),\dots,(u_1,k_1)|,\quad  E^R(h')=|(v_1,l_1),\dots,(v_{h'},l_{h'})), $$ 
 $$E(h')=((u_{h'},k_{h'}),\dots,(u_1,k_1)|(v_1,l_1),\dots,(v_{h'},l_{h'})).$$
 Then if $E, E'\in \cE$, $E\leq E'$ if and only if $sz(E)\geq sz(E')=h'$ and $E(h')\leq E'$.

\subsection{ $L(E^L,J'^L)$ and $R(E^R,J'^R)$}

 For $h'\leq h$,
let $\sigma_L$ and $\sigma_R$ be the permutations of $\{1,2,\dots,h'\}$  such that $u_{\sigma_L(i)}<u_{\sigma_L(i+1)}$ and $v_{\sigma_R(i)}<v_{\sigma_R(i+1)}$.
Let $L(E^L,J'^L)=L(E,J')$ and $R(E^R,J'^R)=R(E,J')$ be the smallest non-negative integers $i_0$ and $j_0$ such that 
$u'_{i}\geq u_{\sigma_L(i-i_0)}$, $i_0< i \leq h'$ and  $v'_{j}\geq v_{\sigma_L(j-j_0)}$, $j_0<j\leq h'$, respectively.

Then $$L(E,J')=L(E(L),J'(L)),\quad R(E,J')=R(E(R),J'(R)), $$
$$L(E^L,J'^L)=L(E^L(h'),J'^L),\quad R(E^R,J'^R)=R(E^R(h'),J'^L),$$ $$L(E,J')=L(E(h'),J'), \quad R(E,J')=R(E(h'),J').$$

The following lemma is obvious.
\begin{lemma}\label{lemma:replace}
For $J''^L=\bpartial^k(u''_{h'},\dots,u''_1)\in\cJ_{h'}^L$, if there are at least $s$ elements in $\{u''_{h'},\dots, u''_1\} $ from the set  $\{u'_{h'},\dots, u'_1\} $, then 
$L(E^L,J''^L)\geq L(E^L,J'^L)-h'+s$;  \\
For $J''^R=\bpartial^k(v''_{1},\dots,v''_{h'})\in\cJ_{h'}^R$, if there are at least $s$ elements in $\{v''_{h'},\dots, v''_1\} $ from the set  $\{v'_{h'},\dots, v'_1\} $, then 
$R(E^R,J''^R)\geq R(E^R,J'^R)-h'+s$.

\end{lemma}
\subsection{A criterion for $J'$ to be greater than $E$}
We say that $J'^L$ is greater than $E^L$ if there is an element $E'^L\in\cE(J'^L)$ with $E^L(h')\leq E'^L$. Similarly, we say that
 $J'^R$ is greater than $E^R$ if there is an element $E'^R\in\cE(J'^R)$ with $E^R(h')\leq E'^R$. We say that
$J'$ is greater than $E$ if there is an element $E'\in\cE(J')$ with $E\leq E'$. Then $J'$  is greater than $E$  if and only if $J'$  is greater than $E(h')$.  

The following lemma  is from \cite{LS1}.
\begin{lemma} \label{lemma:critgreat0}
	$J'$ is greater than $E$ if and only if 
	$wt(J')-wt(E(h'))\geq L(E,J')+R(E,J')$.
\end{lemma}
The next result is analogous.
\begin{lemma} \label{lemma:critgreat}
	$J'^L$ is greater than $E^L$ if and only if 
	$wt(J'^L)-wt(E^L(h')) \geq L(E^L,J'^L)$ . Similarly,
	$J'^R$ is greater than $E^R$ if and only if 
	$wt(J'^R)-wt(E^R(h'))\geq R(E^R,J'^R)$.
\end{lemma}
\begin{proof} 
	$J'^L$ is greater then $E^L$ if and only if there is an $E'\in \cE(J'^L)$ with $E^L\leq E'$. Let $$E_0=|(h,0),\dots,(2,0),(1,0)),\qquad E_1=F(E^L,E_0),\qquad E_2=F(E',E_0).$$ Then $E^L\leq E'$ if and only if $E_1\leq E_2$. We have $wt(E_1(h'))=wt(E^L(h'))$, $wt(E_2)=wt(J'^L)$, $L(E_1,||E_2||)=L(E^L,J'^L)$ and $R(E_1,E_2)=0$. By Lemma \ref{lemma:critgreat0}, $E_1\leq E_2$ if and only if 
$wt(J'^L)-wt(E^L(h'))\geq L(E^L,J'^L)$.
The proof of the second statement is similar.
\end{proof}

\begin{cor} \label{cor:critgreat1}$J'^L$ is greater than $E^L$ if and only if $||E^L(h')||J'^L$ is standard;
	$J'^R$ is greater than $E^R$ if and only if $||E^R(h')||J'^R$ is standard;
	$J'$ is greater than $E$ if and only if $||E(h')||J'$ is standard.
\end{cor}
\begin{proof} We only prove the third statement, since the proofs of the others are similar.
	By Lemma \ref{lemma:critgreat}, $J'$ is greater than $E$ if and only if $wt(J')-wt(E(h'))\geq L(E,J')+R(E,J')$ and $||E(h')||J'$ is standard if and only if $wt(J')-wt(E(h'))\geq L(E(h'),J')+R(E(h'),J')=L(E,J')+R(E,J')$.
\end{proof}

\subsection{The \lq\lq largest" property }	
Let
\begin{eqnarray*}
	\cW^L_s(E^L,J'^L)&=&\{J^L=\bpartial^k(u'_{i_s},\dots, u'_{i_1}|| \ 1\leq i_l\leq h',\ J^L \text{ is greater than } E^L\};\\	
\cW^R_s(E^R,J'^R)&=&\{J^R=\bpartial^k|v'_{j_1},\dots, v'_{j_s})| \ 1\leq j_l\leq h',\ J^R \text{ is greater than } E^R\};\\
\cW_s(E,J')&=&\{J=\bpartial^k(u'_{i_s},\dots, u'_{i_1}|v'_{j_1},\dots, v'_{j_s})| \ 1\leq i_l,j_l\leq h',\  J \text{ is greater than } E\}.
\end{eqnarray*}

Let $W_s^L(E,J')=W_s^L(E(L),J'(L))$ and $W_s^R(E,J')=W_s^R(E(R),J'(R)).$

\begin{lemma}\label{cor:compare}
If $E'^L$ is the largest element in $\cE(J'^L)$ such that $E^L\leq E'^L$, then for $s<h'$,
	$||E'^L(s)||$
	is the smallest element in $\cW_s^L(E^L,J'^L)$.
If $E'^R$ is the largest element in $\cE(J'^R)$ such that $E^R\leq E'^R$, then for $s<h'$,
$||E'^R(s)||$
is the smallest element in $\cW_s^R(E^R,J'^R)$.
If $E'$ is the largest element in $\cE(J')$ such that $E\leq E'$, then for $s<h'$,
$||E'(s)(L)||$
is the smallest element in $\cW_s^L(E,J')$ and $||E'(s)(R)||$
is the smallest element in $\cW_s^R(E,J')$, so $||E'(s)||$ is the smallest element in $\cW_s(E,E')$.
	\end{lemma}
\begin{proof} The third statement is Lemma 5.4 of \cite{LS1}.
	We use a similar method to prove the first statement. The proof of the second statement is similar and is omitted.
	Assume $$E'^L=(( u'_{h'},k_{h'}),\dots,(u'_{2},k_2),( u'_{1}, k_1)|.$$
	For $s<h'$, let $J^L_s$ be the smallest element in $\cW_s^L(E^L,J'^L)$. Let
	$$E_s=(( u'_{i_s},\tilde k_{s}),\dots,(u'_{i_2},\tilde k_2),( u'_{i_1},\tilde k_1)|$$
	 be the largest element in $\cE(J^L_s)$ such that $E^L(s)\leq E _s$. 
	 
	 Assume $l$ is the largest number such that $(u'_j, k'_j)=(u'_{i_j},\tilde k_j)$ for $j<l$. 
	 If $l=s+1$, then $E_s=E'^L(s)$, so $||E'^L(s)||$ is the smallest element in $\cW_s^L(E^L,E'^L)$. 
	 Otherwise $l\leq s$, then $i_l\geq l$ and $(u'_{i_l},\tilde k_l)\neq (u'_l,k'_l)$. If $i_l=l$, by the maximality of $E'^L$ and the minimality of $J_s^L$, we must have $(u'_{i_l},\tilde k_l)= (u'_l,k'_l)$. This is a contradiction, so $i_l>l$. 
	 
	 If $(u'_{i_l},\tilde k_l)<(u'_l,k'_l))$, then
	$(u'_{l}, k'_l+k'_{i_l}-\tilde k_{l})> (u'_{i_l}, k'_{i_l})$.
	Let $E''$ be the element in $\cE_{h'}^L(J'^L)$  by replacing $( u'_{l},k'_l)$ and $(u'_{i_l},k'_{i_l})$ in $E'^L$
	by $(u'_{i_l},\tilde k_{l})$ and   $(u'_{l}, k'_l+k'_{i_l}-\tilde k_{l})$, respectively. We have $E'^L\prec E''$ and $E^L(h')\leq E''$.
	But $E'^L\neq E''$ is the largest element in $\cE(||E'^L||)$ such that $E^L\leq E'^L$, a contradiction.

	Assume $(u'_{i_l},\tilde k_l)>(u'_l,k'_l)$. If $l\notin\{i_1,\dots,i_s\}$, replacing $(u'_{i_l},\tilde k_l)$ in $E_s$ by $(u'_l,k'_l)$, we get $E'_s$ with $E^L(s)\leq E'_s$ and $||E'_s||\prec J_s$. This is impossible since $J_s\neq ||E'_s||$
	is the smallest element in $\cW_s(E^L,E'^L) $. If $l=i_j \in\{i_1,\dots,i_s\}$,
	$(u'_{i_l}, \tilde k_l+\tilde k_{j}- k'_{l})> (u'_{i_j}, \tilde k_{j})$.
	Let $E'_s$ be the element of $\cE(J_s)$ obtained by replacing $(u'_{i_l},\tilde k_l)$ and $(u'_{i_j}, \tilde k_{j})$ in $E'_s$
	by $(u'_l,k'_l)$ and   $(u'_{i_l}, \tilde k_l+\tilde k_{j}- k'_{l})$, respectively. We have $E_s\prec E'_s$ and $E^L\leq E'_s$.
	But $E'_s\neq E_s$ is the largest element in $\cE(J_s)$ such that $E^L\leq E'_s$, a contradiction.
\end{proof}
\begin{cor}\label{cor:lrnumber1}
If $E'^L$ is the largest element in $\cE(||E'^L||)$ such that $E^L\leq E'^L$, then for $s<h'$,
	$$L(E,||E'(s)||)=wt(E'^L(s))-wt(E^L(s));$$
	if $E'^R$ is the largest element in $\cE(||E'^R||)$ such that $E^R\leq E'^R$, then for $s<h'$,
	$$R(E^R,||E'^R(s)||)=wt(E'^R(s))-wt(E^R(s));$$
If $E'$ is the largest element in $\cE(||E'||)$ such that $E\leq E'$, then for $s<h'$,
$$L(E,||E'(s)||)+R(E,||E'(s)||)=wt(E'(s))-wt(E(s)).$$	
\end{cor}
\begin{proof}
We only prove the first statement since the proofs of the other statements are similar. Since $E^L\leq E'^L$, $E^L\leq E'^L(s)$. 
By Lemma \ref{cor:compare}, $||E'(s)||$ is the smallest element in $\cW^L_s(E^L,J'^L)$. By Lemma \ref{lemma:critgreat}, $L(E,||E'(s)||)= wt(E'^L(s))-wt(E^L(s))$.
\end{proof}
\begin{cor}\label{cor:lrnumber2}
	If $E'^L$ is the largest element in $\cE(||E'^L||)$ such that $E^L\leq E'^L$, then for $s<h'$, and any
	$J^L\in \cW_s^L(E^L,E'^L)$,
		$$L(E^L,||E'^L(s)||)\leq L(E^L,J^L);$$
		If $E'^R$ is the largest element in $\cE(||E'^R||)$ such that $E^R\leq E'^R$, then for $s<h'$, and 	$J^R\in \cW_s^R(E^R,E'^R)$,
	$$R(E^R,||E'^R(s)||)\leq R(E^R,J^R);$$	
If $E'$ is the largest element in $\cE(||E'||)$ such that $E\leq E'$, then for $s<h'$ and any	$J\in \cW_s(E,E')$,
	$$L(E,||E'(s)||)\leq L(E,J), \quad R(E,||E'(s)||)\leq R(E,J).$$
\end{cor}
\begin{proof}
	We only prove the first statement since the proofs of the others are similar. \\
	Assume there is  $J^L=\bpartial^k(u'_{i_s},\dots, u'_{i_1}| \in \cW_s^L(E^L,E'^L)$ with $L(E^L,||E'^L(s)||)> L(E^L,J^L)$. By  Lemma \ref{lemma:critgreat}, $k=wt(J^L)$ can be $wt(E^L(s))+L(E^L, J^L)$. By Lemma \ref{cor:lrnumber1},
	$$L(E^L,||E'^L(s)||)=wt(E'^L(s))-wt(E^L(s)).$$ So $wt(J^L)<wt({E'}^L(s))$.
	But by Lemma \ref{cor:compare}, $||E'^L(s)||$ is the smallest element in $\cW_s^L(E^L,E'^L)$, we have $wt(E'^L(s))\leq wt(J^L)$.  Contradiction!
	\end{proof}
We have the following lemma from \cite{LS1}.
\begin{lemma}\label{lem:lrnumber}
	Let $E_i=((u^i_{h_i},k^i_{h_i}),\dots,(u^i_1,k^i_1)|(v^i_1,l^i_1),\dots,(v^i_{h_i},l^i_{h_i}))$, $i=a,b$. Suppose that $E_b\leq E_a$, and that $E_a$ is the largest element in $\cE(||E_a||)$ such that $E_b\leq E_a$. Let $1\leq h<h_a$ and let $\sigma_i$, $\sigma'_i$ be permutations of $\{1,\dots, h\}$, such that $u^i_{\sigma_i(1)}<u^i_{\sigma_i(2)}<\cdots<u^i_{\sigma_i(h)}$ and $v^i_{\sigma'_i(1)}<v^i_{\sigma'_i(2)}<\cdots<v^i_{\sigma'_i(h)}$. Let $v_1',\dots,v'_{h_a}$ be a permutation of $v^a_1,\dots,v^a_{h_a}$ such that $v_1'<v_2'<\cdots<v'_{h_a}$.  Let $u_1',\dots,u'_{h_a}$ be a permutation of $u^a_1,\dots,u^a_{h_a}$ such that $u_1'<u_2'<\cdots<u'_{h_a}$.
	\begin{enumerate}
		\item \label{item:1}
		Assume $u'_{i_2}=v^a_{\sigma(i_1)}$ with $i_2>i_1$, then for any
		$$K=\bpartial^k(u_h,\dots, u_{s+1},u'_{t_s},\dots, u'_{t_1}|v_{1},\dots,v_h),$$
		with $t_1<t_2<\cdots <t_s<i_2$, $L(E_b,K)>L(E_b,||E_a(h)||)+s-i_1$. 
		\item \label{item:2}
		Assume $v'_{j_2}=v^a_{\sigma(j_1)}$ with $j_2>j_1$, then for any
		$$K=\bpartial^k(u_h,\dots,u_1|v'_{t_1},\dots,v'_{t_s},v_{s+1},\dots,v_h),$$
		with  $t_1<t_2<\cdots <t_s<j_2$, $R(E_b,K)>R(E_b,||E_a(h)||)+s-j_1$. 
	\end{enumerate}
\end{lemma}

The following lemmas are obvious:
\begin{lemma}\label{lemma:order1}If $J_1\prec J_2\prec \cdots \prec J_n$, $\sigma$ is a permutation of $\{1,\dots ,n\}$, then
	$$J_1J_2\cdots J_n\prec J_{\sigma(1)}\cdots J_{\sigma(n)}.$$
\end{lemma}
\begin{lemma} \label{lemma:order2}If $K_1K_2\cdots K_k \prec J_1\cdots J_l$, then
	$$K_1K_2\cdots K_{s-1}JK_s\cdots K_k \prec J_1\cdots J_{s-1}JJ_s\cdots J_l. $$
\end{lemma}

\section{Proof of Lemma \ref{lem:base20}}\label{sec:base}
In this section we prove Lemma \ref{lem:base20}. This is generalization of Lemma 2.3 of \cite{LS1}, which is the case when $J_1,\dots, J_b$ lie in the subset $\cJ_h\subseteq \cJ_h^+$. We proceed by induction, proving first the case $b=2$, and then assuming the result for $b-1$. We will repeatedly make use of the technical statements from the previous section, namely Lemmas \ref{lemma:critgreat0}, \ref{lemma:critgreat}, and \ref{lem:lrnumber}, and Corollaries \ref{cor:critgreat1},  \ref{cor:lrnumber1}, and \ref{cor:lrnumber2}. The proof is more complicated here because we need to consider the following cases separately:
\begin{enumerate}
\item $J_1,\dots, J_{b-1} \in \cJ^L_h$, 
\item $J_1,\dots,J_{b-1} \in \cJ^R_h$, 
\item $J_1,\dots, J_c\in \cJ^L_h$ and $J_{c+1},\dots, J_{b-1} \in \cJ^R_h$, for some $1\leq c < b-1$,
\item $J_{b-1} \notin \cJ^L_h \cup \cJ^R_h$. 
\end{enumerate}

By Lemma \ref{lemma:order1}, we can assume each monomial is expressed as an ordered product $J_{1}J_{2}\cdots J_{b}$ with $J_{a}\prec J_{a+1}$.
For $\alpha \in \cM(\cJ_h^+)$,
let $$\R_h^+(\alpha)=\{\sum c_i\beta_i\in\R_h^+|c_i\in\mathbb Z, \beta_i\in \cM(\cJ_h^+),\beta_i\prec \alpha, \beta_i\neq \alpha\},$$
the space of linear combinations of elements preceding $\alpha$ in $\cM(\cJ^+_h)$ with integer coefficients. 
The following lemma is from \cite{LS1}.
\begin{lemma}\label{lemma:straight0}
	Let $E\in \cE$, $J_a$ and $J_b$ in $\cJ$ with $J_a\prec J_b$, and suppose that $E_a$ is the largest element in $\cE(J_a)$ such that $E\leq E_a$. If $J_b$ is not greater than $E_a$, then $J_aJ_b=\sum K_if_i$ with $K_i\in \cJ$, $f_i\in \R$ such that $K_i$ is either smaller than $J_a$, or $K_i$ is not greater than $E$. 
\end{lemma} 

\begin{lemma}\label{lemma:case1+}
	Assume $J_1J_2$ is not standard.  If $J_1, J_2\in \cJ_h^L$, $J_1J_2$ can be written as a linear combination of ordered products of elements of $\cJ_h^L$ preceding $J_1$ with integer coefficients.\\
	If $J_1, J_2\in \cJ_h^R$ is not standard, $J_1J_2$ can be written as a linear combination of ordered products of elements of $\cJ_h^R$ preceding $J_1$ with integer coefficients.
\end{lemma}
\begin{proof} We only prove the first statement, since the proof of the second statement is similar. Assume $J_i=\bpartial^{n_i}(u^i_{h},\dots,u^i_2,u^i_1|$ , $i=1,2$. Let $n=n_2- n_1\geq 0 $. Since $J_1J_2$ is not standard,
	by Lemma \ref{lemma:critgreat},
	there is an integer $i>n$, $u^2_{i}<u^1_{i-n}$. By Equation (\ref{eqn:relation3+}),
	$$\sum \epsilon a_{k} \left(
	\begin{array}{c}
	\bpartial^{m-k} (\underline{u^1_h,\dots,u^1_{i-n}},u^1_{i-n-1},\dots,u^1_1| \\
	\bpartial^{k} (u^2_{h},\dots,u^2_{i+1},\underline{u^2_{i},\dots,u^2_1}|
	\end{array}
	\right)
	=0.
	$$
	Here $m=n_1+n_2$,  $a_{n_2-l}=\delta_{0,l}$ for $0 \leq l\leq n$. In the above equation:
	\begin{itemize}
		\item
		All the terms with $k=n_2$ precede $J_1$ in the lexicographic order except $J_1J_2$ itself;
		\item
		All the terms with $k=n_2-1,\dots,n_1$ vanish since $a_{k}=0$;
		\item
		All the terms with $k=n_2+1,\dots, m$ precede $J_1$ since the weight of the upper $\bpartial$-list is $m-k<n_1$.
		\item
		All the terms with $k=0,\dots, n_1-1$ precede $J_1$ after exchanging the upper $\bpartial$-list and the lower $\bpartial$-list since the weight of the lower $\bpartial$-list is $k<n_1$.
	\end{itemize}
	Thus, $J_1J_2$ can be written as a linear combination of ordered products of elements of $\cJ_h^L$ preceding $J_1$ with integer coefficients.
\end{proof}

\begin{proof}[Proof of Lemma \ref{lem:base20}.]
	Since $J_1$ is always standard, we can asssume $b\geq 2$. Assume $J_1\cdots J_b$ is not standard.
	
	 If $J_b\in J_h^R\cup J_h^L$, then $J_i\in J_h^R\cup J_h^L$ for $1\leq i\leq b$.
	By Lemma \ref{lemma:case1+} and Corollary \ref{cor:critgreat1}, it is easy to see that $J_1\cdots J_b \in \R_h^+(J_1\cdots J_{b-1})$.
	
	If $J_b=\bpartial^{n_b} (u^b_{h_b},\dots,u^b_1|v^b_1,v^b_2,\dots,v^b_{h_b})\notin J_h^R\cup J_h^L$, then $h_b<h$. We show the lemma by induction on $b$.
	
	For $b=2$, $J_1J_2$ is not standard by assumption.
	\begin{enumerate}
		\item If $J_1=\bpartial^{n_1} (u^1_h,\dots,u^1_1|$, 
	there is an integer $i>n_2$ such that $u^2_i<u^1_{i-n_2}$. Let $m=n_1+n_2$, by Equation (\ref{eqn:relation1+}),
		$$\sum \epsilon a_{k} \left(
		\begin{array}{ccc}
		\bpartial^{m-k} (\underline{u^1_h,\dots,u^1_{i-n_2}},u^1_{i-n_2-1},\dots,u^1_1&|& \\
		\bpartial^{k} (u^2_{h_2},\dots,u^2_{i+1},\underline{u^2_{i},\dots,u^2_1}&|&v^2_1, v^2_2,\dots,v^2_{h_2})
		\end{array}
		\right)
		=0
		$$
		with $a_{n_2-l}=\delta_{0,l}$ for $0\leq l\leq n_2$.
		In the above equation,
		\begin{enumerate}
			\item
			All the terms with $k=n_2$ precede $J_1$ in the lexicographic order except $J_1J_2$ itself;
			\item
			All the terms with $k=0,\dots,n_2-1$ vanish since $a_{k}=0$;
			\item
			All the terms with $k=n_2+1,\dots, m$ precede $J_1$, since the weight of the upper $\bpartial$-list is $m-k<n_1$.
		\end{enumerate}
		Thus, $J_1J_2\in \R_h^+(J_1)$.
		\item If $J_1=\bpartial^{n_1}|v^1_1,\dots,v^1_h)$,  similar to case (1), we can show $J_1J_2\in \R_h^+(J_1)$.
		\item If $J_1\notin \cJ_h^R\cup \cJ_h^L$, then $J_1\in \cJ$.  By Lemma 7.1 of \cite{LS1},
		$J_1J_2\in \R(J_1)$, so  $J_1J_2\in \R_h^+(J_1)$.
	\end{enumerate}

For $b>2$, let $a=b-1$. By induction, the lemma is true for $J_1\cdots J_{a}$: if  $J_1\cdots J_{a}$ is not standard, then  $J_1\cdots J_{a}\in\R_h^+(J_1\cdots J_{a-1})$.
So we can assume $J_1\cdots J_{a}$  and $J_1\cdots J_{b-2}J_b$ are standard by Lemma \ref{lemma:order2}. Let $E_1\cdots E_{a}\in\cS\cM(\cJ_h^+)$ be the standard ordered product of elements of $\cE_h^+$ corresponding to $J_1\cdots J_{a}$. To show $J_1\cdots J_b\in \R_h^+(J_1\cdots J_a)$, we only need to show \\
$(*)$ \quad \textsl{
$J_aJ_b=\sum K_if_i$ with $K_i\in \cJ_h^+$ and $f_i\in \R_h^+$ such that either 
$K_i$ is smaller than $J_a$ or $J_1\cdots J_{b-2}K_i$ is not standard.}\\ 
In the first case $J_1\cdots J_{b-2}K_if_i\in \R_h^+(J_1\cdots J_a)$  and in the second case $J_1\cdots J_{b-2}K_i\in\R_h^+(J_1\cdots J_{b-2})$ by induction, then 
$J_1\cdots J_{b-2}K_if_i\in\R_h^+(J_1\cdots J_{a})$.

Now let us show the statement $(*)$. The proof is divided into four cases.\\
	\textsl{Case 1:} $J_{a}\in\cJ_h^L$, then $J_i\in\cJ_h^L$ for $1\leq i\leq a$.

	Assume $E_i=((u^i_h,k^i_h),\dots, (u^i_1,k^i_1)|$. Let $m_i=wt(E_i(h_b))$ and  $n_i=wt(E_i)$. Let $\sigma_a$ be the permutations of ${1,2,\dots, h_b}$
	such that $u^a_{\sigma_a(j)}<u^a_{\sigma_a(j+1)}$ for $1\leq j\leq h_{b-1}$. Let
	$u'_{h},\dots,u'_1$ be a permutation of $u^a_{h},\dots,u^a_1$ such that $u'_{i}\leq u'_{i+1}$. Let $m=n_a+n_b$. 
		\begin{enumerate}
		\item If $n_b<m_a$. By Equation (\ref{eqn:relation1+}),
		 \begin{equation}\label{eqn:91}
		\sum_{0\leq i<h_b}\sum_{\sigma}\frac {(-1)^{i}\sign(\sigma)}{i!(h_b-i)!}\sum_{k=0}^{m} \epsilon a^{i}_{k}
	\left(
		\begin{array}{ccc}
		\bpartial^{m-k}(\underline{u_{h}^a,\dots,u_{i+1}^a},u^b_{\sigma{(i)}},\dots,u_{\sigma(2)}^b,u^b_{\sigma(1)}&|& \\
		\bpartial^k(\underline{u_{\sigma(h_b)}^b,\dots,u^b_{\sigma{(i+1)}},u^a_{i},\dots,u^a_{1}}&|& v^b_{1},\dots,v^b_{h_b})
		\end{array}
		\right)=0.
		\end{equation}

		Here  $a_k^{i}$ are integers and $a_{n_b-l}^{i}=\delta_{0,l}$ for $0 \leq l<h_b -i$, and $\sigma$ are permutations of $\{1,\dots,h_b\}$.
		In the above equation:
		\begin{enumerate}
			\item
			All the terms with $k=n_b+1,\dots, n$ precede $J_a$ since the weight of the upper $\bpartial$-list is $m-k<n_a$.
			\item
			The terms with $k=n_b$ are $J_aJ_b$
			and the terms with the lower $\bpartial$-lists
			$$K_0=\bpartial^{n_b}(u_{i_{h_b}}',\dots,u_{i_2}',u'_{i_1}|v^b_{1},v^b_{2},\dots,v^b_{h_b}).$$
			All of the other terms cancel. Let 
				$$K'_0=\bpartial^{n_b}(u_{i_{h_b}}',\dots,u_{i_2}',u'_{i_1}|.$$
			By Corollary \ref{cor:lrnumber1} and \ref{cor:lrnumber2}
			$$L(E_{b-2}, K'_0))\geq L(E_{b-2}, ||E_a(h_b)||))=wt(E_a(h_b))-wt(E_{b-2}(h_b))>n_b-m_{b-2}.$$
			By Lemma \ref{lemma:critgreat}, $K'_0$ is not greater than $E_{b-2}$.
		So $J_1\cdots J_{b-2}K_0$ is not standard.
			\item
			The terms with $k<n_b$ in Equation (\ref{eqn:91}) vanish unless $h_b-i\leq  n_b-k$. In this case the lower $\bpartial$-lists of the terms are
			$$K_1=\bpartial^k(\underline{u_{\sigma(h_b)}^b,\dots,u^b_{\sigma{(i+1)}},u'_{s_i},\dots,u'_{s_1} }|v^b_{1},v^b_{2},\dots,v^b_{h_b}).$$
			Let 	$$K'_1=\bpartial^k(\underline{u_{\sigma(h_b)}^b,\dots,u^b_{\sigma{(i+1)}},u'_{s_i},\dots,u'_{s_1}}|.$$
			By Lemma \ref{lemma:replace},
		$$L(E_{b-2}, ||K'_1||)\geq L(E_{b-2},|| E_a(h_b)||)-(h_b-i)=m_a-m_{b-2}-(h_b-i)>k-m_{b-2}.$$
			By Lemma \ref{lemma:critgreat}, $K'_1$ is not greater than $E_{b-2}$.
		So $J_1\cdots J_{b-2}K_1$ is not standard.	
		\end{enumerate}
		By the above analysis, $J_1\cdots J_b\in \R_h^+(J_1\cdots J_a)$.
		\item
		If $n_b\geq m_a$. 
		
		 Let $J'_b=\bpartial^{n_b} (u^b_{h_b},\dots,u^b_1|$ and $i_0=L(E_a, J'_b).$ 
		Since $J_1\cdots J_b$ is not standard, $J'_b$ is not greater than $E_a$.
 By Lemma \ref{lemma:critgreat}, 
 \begin{equation}\label{eqn:numberi0}
 i_0>n_b-m_a\geq 0.
 \end{equation}
  There is $i_0\leq i_1\leq h_b$, such that $u^b_{i_1}<u^a_{\sigma_a(i_1-i_0+1)}$.

		Assume $u'_{i_2}=u^a_{\sigma_a(i_1-i_0+1)}$, then $i_2\geq i_1-i_0+1$.
		By Equation (\ref{eqn:relation1+}),
		$$0=\sum_{\substack{i_1>s\\ i_2-1\geq s}}\sum_{\sigma}\frac {(-1)^{s}\sign(\sigma)}{s!(i_2-1-s)!}\sum_{k=0}^{n}\epsilon a^{s}_{k} $$
		\begin{equation}\label{eqn:relationterms2+0}
		\left( \begin{array}{ccc}
		\bpartial^{m-k}(\underline{u'_{h} ,\dots,u'_{i_2} ,u'_{\sigma(i_2-1)},\dots, u'_{\sigma(s+1)},u_{s}^b,\dots,u^b_{1}}&|&\\
		\bpartial^k(u^b_{h_b},\dots,u^b_{i_1+1},\underline{u^b_{i_1},\dots,u^b_{s+1}},u'_{\sigma(s)},\dots,u'_{\sigma(1)}&|& v^b_1,\dots, v^b_{h_b} )
		\end{array}\right).
		\end{equation}
		Here  $a_k^{s}$ are integers and $a_{n_b-l}^{s}=\delta_{0,l}$ for $0 \leq l<i_1-s$, and
		$\sigma$ are permutations of $\{1,\dots,i_2-1\}$.
		$$0=\sum_{\substack{i\geq i_1> s\\ i_2-1\geq s}}\sum_{\sigma, \sigma'}\frac {(-1)^{i+s}\sign(\sigma')\sign(\sigma)}{(i-i_1)!(h_b-i)!s!(i_2-1-s)!}\sum_{k=0}^{n}\epsilon a^{i,s}_{k} $$
		
		\begin{equation}\label{eqn:relationterms2+1}
		\left( \begin{array}{ccc}
		\bpartial^{n-k}(u^b_{\sigma'(h_b)},\dots,\underline{u^b_{\sigma'(i)},\dots u^b_{\sigma'(i_1+1)},u^b_{i_1}, \dots u^b_{1},u'_{h_a}, \dots  u'_{h_b+1}}
		&|&\\
		\bpartial^k(\underline{u'_{h_b},\dots, u'_{i_2},u'_{\sigma(i_2-1)}, \dots, u'_{\sigma(s+1)}},u'_{\sigma(s)}, \dots, u'_{\sigma(1)}&|&
		v^b_{1},v^b_{2} ,\dots v^b_{h_b})
		\end{array}\right).
		\end{equation}

		Here $a_k^{i,s}$ are integers, $a_{n_b-l}^{i,s}=\delta_{0,l}$ for $0 \leq l<i-s$,
		$\sigma$ are permutations of $\{1,\dots,i_2-1\}$, and $\sigma'$ are permutations of $\{i_1+1,i_1+2,\dots, h_b\}$.

		If $i_2> i_1-i_0+1$, we use Equation (\ref{eqn:relationterms2+0}) and if $i_2=i_1-i_0+1$ we use Equation (\ref{eqn:relationterms2+1}).
		
		In the above Equations:
		\begin{enumerate}
			\item
			All the terms with $k=n_b+1,\dots, n$ precede $J_a$ since the weight of the upper $\bpartial$-list is $m-k<n_a$.
			\item
			The terms with $k=n_b$ are $J_aJ_b$, the terms with the upper $\bpartial$-lists preceding $J_a$ (the upper $\bpartial$-lists are the $\bpartial$-lists given
			by replacing some $u'_{l}$, $l\geq i_2$ in $J_a$ by some $u^b_{k}$, $k\leq i_1$),
			and the terms with the lower $\bpartial$-lists
			$$K_0=\bpartial^{n_b}(u_{h_b}^b,\dots,u_{i_1+1}^b, u'_{\sigma_1(i_1)},\dots u'_{\sigma_1(1)}|v^b_{1}, v^b_{2},\dots,u^b_{h_b}).$$
			All of the other terms cancel. Let 
				$$K'_0=\bpartial^{n_b}(u_{h_b}^b,\dots,u_{i_1+1}^b, u'_{\sigma_1(i_1)},\dots u'_{\sigma_1(1)}|.$$
			
			If $i_2-1<i_1$, the terms of the form $K_0$ do not appear in Equations \eqref{eqn:relationterms2+0} and \eqref{eqn:relationterms2+1}.
			Otherwise  $i_2-1\geq i_1$, by Lemma \ref{lem:lrnumber},
			$$L(E_{b-2}, K'_0)>L(E_{b-2}, ||E_a(h_b)||)+ i_1-(i_1-i_0+1).
			$$
			By Corollary \ref{cor:lrnumber1},
			\begin{equation}\label{eqn:LnumberE}
			L(E_{b-2}, ||E_a(h_b)||)=m_a-m_{b-2}.
			\end{equation}
			Then by Equation (\ref{eqn:numberi0}),
			$$L(E_{b-2}, K'_0)\geq m_a-m_{b-2}+i_0>n_b-m_{b-2}.
        	$$		
			
        	By Lemma \ref{lemma:critgreat},  $K'_0$ is not greater than $E_{b-2}$.  So there is no $E_b'\in \cE(K_0)$ such that $E_{b-2}\leq E_b'$, and $J_1\cdots J_{b-2}K_0$ is not standard.	
		
			\item If $i_2>i_1-i_0+1$,
			the terms with $k<n_b$ in Equation (\ref{eqn:relationterms2+0}) vanish unless $i_1-s\leq  n_b-k$. In this case the lower $\bpartial$-lists of the terms are
			$$K_1=\bpartial^k(u^b_{h_b},\dots,u^b_{i_1+1},\underline{u^b_{i_1},\dots,u^b_{s+1}},u'_{\sigma(s)},\dots,u'_{\sigma(1)}|v^b_1,\dots, v^b_{h_b} ).$$
			The underlined elements can be any underlined elements in Equation (\ref{eqn:relationterms2+0}).
			Let 
				$$K'_1=\bpartial^k(u^b_{h_b},\dots,u^b_{i_1+1},\underline{u^b_{i_1},\dots,u^b_{s+1}},u'_{\sigma(s)},\dots,u'_{\sigma(1)}|.$$
				By Lemma \ref{lem:lrnumber},
				$$L(E_{b-2}, K'_1)>L(E_{b-2}, ||E_a(h_b)||)+ s-(i_1-i_0+1).
				$$
			Then by Equations (\ref{eqn:LnumberE}) and (\ref{eqn:numberi0}),
			$$L(E_{b-2}, K'_0)\geq m_a-m_{b-2}+i_0-(i_1-s)>k-m_{b-2}.
			$$		
			
			By Lemma \ref{lemma:critgreat},  $K'_1$ is not greater than $E_{b-2}$.  So there is no $E_b'\in \cE(K_1)$ such that $E_{b-2}\leq E_b'$, and  $J_1\cdots J_{b-2}K_1$ is not standard.	
			
			\item If $i_2=i_1-i_0+1$, the terms with $k<n_b$ in Equation (\ref{eqn:relationterms2+1}) vanish unless $i-s\leq  n_b-k$. In this case the lower $\bpartial$-lists of the terms are
			$$K_1=  \bpartial^k(\underline{u'_{h_b},\dots, u'_{i_2},u'_{\sigma(i_2-1)}, \dots, u'_{\sigma(s+1)}},u'_{\sigma(s)}, \dots, u'_{\sigma(1)}|
			v^b_{1},v^b_{2} ,\dots , v^b_{h_b})$$
			Underlined elements can be any underlined elements in Equation (\ref{eqn:relationterms2+1}).
			Let $$K'_1=  \bpartial^k(\underline{u'_{h_b},\dots, u'_{i_2},u'_{\sigma(i_2-1)}, \dots, u'_{\sigma(s+1)}},u'_{\sigma(s)}, \dots, u'_{\sigma(1)}| $$
		Let
			$$K'_0=\bpartial^{m_a-1}(u'_{l_{h_b}},\dots, u'_{l_{i_2}},u'_{i_2-1}, \dots, u'_{1}|.
		$$
			Here $i_2\leq l_{i_2}< l_{j_2+1}<\cdots< l_{h_b}\leq h_a$. By Lemma \ref{lemma:replace}, there is $K'_0$ with
		\begin{equation}\label{eqn:LK1-5}
		L(E_{b-2},K'_1)\geq R(E_{b-2},K'_0)-(i_2-1-s)-(i-i_1)
		\end{equation}
	 since in $K'_1$ the number of $u^b_i$ with $u^b_i>v'_{i_2}$ is at most $i-i_1$. By  Corollary \ref{cor:lrnumber1} and \ref{cor:lrnumber2},
		\begin{equation}\label{eqn:LK1-5+}
		R(E_{b-2},K_0')\geq R(E_{b-2},||E_a(h_b)||)=m_a-m_{b-2}.
		\end{equation}
		 So by Equations \eqref{eqn:LK1-5},\eqref{eqn:LK1-5+}, and \eqref{eqn:numberi0},
		$$L(E_{b-2},K'_1)\geq m_a-m_{b-2}+i_0-(i-s)>k-m_{b-2}.$$
			By Lemma \ref{lemma:critgreat},  $K'_1$ is not greater than $E_{b-2}$.  So there is no $E_b'\in \cE(K_1)$ such that $E_{b-2}\leq E_b'$, and $J_1\cdots J_{b-2}K_1$ is not standard.
		\end{enumerate}
	\end{enumerate}
	\textsl{Case 2:}  $J_i\in\cJ_h^R$ for $1\leq i\leq a$, the proof is similar to the case of $J_a\in \cJ_h^L$.\\
	\\
	\textsl{Case 3:} $J_1,\dots, J_c\in \cJ_h^L$ and $J_{c+1},\dots,J_a\in \cJ_h^R$, $1\leq c<a$.
	
	Let $E_1\cdots E_{a}$ be the standard ordered product of elements of $\cE_h^+$ corresponding to $J_1\cdots J_{a}$.
	Assume $E_i=((u^i_h,k^i_h),\dots, (u^i_1,k^i_1)|$, for $1\leq i\leq c$ and $E_i=|(v^i_1,k^i_1),\dots, (v^i_h,k^i_h))$, for $c< i\leq a$.
	Let $m_i=wt(E_i(h_b))$ and $n_i=wt(E_i)$. Let $\sigma_i$ be a permutation of ${1,2,\dots, h_b}$
	such that $u_{\sigma_i(j)}<u_{\sigma_i(j+1)}$ or $v_{\sigma_i(j)}<v_{\sigma_i(j+1)}$ for $1\leq j\leq h_{b-1}$. Let
	$v'_{h_a},\dots,v'_1$ be a permutation of $v^a_{h},\dots,v^a_1$ such that $v'_{i}\leq v'_{i+1}$. Let $m=n_a+n_b$.
 Let $k_0=m_c+L(E_c,J_b(L))$ and $i_0=R(E_a,J_b(R))$.
	Then by Lemma \ref{lemma:critgreat},
	\begin{equation}\label{eqn:numberi0-1}
	 i_0>n_b-k_0-m_a.
	\end{equation}
By assumption, $J_1\cdots J_{b-2}J_b$ is standard. By Lemma \ref{lemma:critgreat}, $k_0$ is the smallest number such that there is $E\in\cE(J_b)$
	with $E_c\leq E$. Then $k_0\leq n_b$. Now the proof is similar to \textsl{Case 1}.
	
	\begin{enumerate}
		\item If $n_b-k_0<m_a$, then $m_a>0$ and $c<b-2$.
		By Equation \eqref{eqn:relation2+}, \begin{equation}0=\sum_{0\leq i<h_b}\sum_{\sigma}\frac {(-1)^{i}\sign(\sigma)}{i!(h_b-i)!}\sum_{k=0}^{m} \epsilon a^{i}_{k}
		\label{eqn:relationterm61} \left(
		\begin{array}{ccc}
		\bpartial^{m-k}&|&v^b_{\sigma(1)},v_{\sigma(2)}^b,\dots,v_{\sigma(i)}',\underline{v_{i+1}',\dots,v'_{h}})\\
		\bpartial^k(u^b_{h_b},\dots,u^b_{1}&|&\underline{ v'_{1},\dots,v'_{i},v^b_{\sigma(i+1)},\dots,v^b_{\sigma(h_b)}})
		\end{array}
		\right)
		\end{equation}
		Here  $a_k^{i}$ are integers and $a_{n_b-l}^{i}=\delta_{0,l}$ for $0 \leq l<h_b -i$, and $\sigma$ are permutations of $\{1,\dots,h_b\}$.
		In the above equations:\\
		\begin{enumerate}
			
			\item
			All the terms with $k=n_b+1,\dots, n$ precede $J_a$, since the weight of the upper $\bpartial$-list is $m-k<n_a$.
			\item
			The terms with $k=n_b$ are $J_aJ_b$
			and the terms in Equation (\ref{eqn:relationterm61}) with the lower $\bpartial$-lists
			$$K_0=\bpartial^{n_b}(u_{h_b}^b,\dots,u_{2}^b,u^b_{1}|v'_{i_1},v'_{i_2},\dots,v'_{i_{h_b}}).$$
			All of the other terms cancel.

			By Lemma \ref{lemma:critgreat}, if there is $E\in \cE(K_0)$ such that $E_{c}\leq E$ and $E_{b-2}\leq E$, we must have $wt(E(L))\geq k_0$ and 
			$$wt(E(R))-wt(E_{b-2}(h_b))\geq L(E_{b-2},||E(R)||).$$
			By Corollary \ref{cor:lrnumber2} and Corollary \ref{cor:lrnumber1},
			\begin{equation}\label{eqn:456}
			L(E_{b-2},||E(R))||\geq L(E_{b-2},||E_a(h_b)||)=m_a-m_{b-2}.
			\end{equation}
		 Then 
			$$n_b=wt(E(L))+wt(E(R))\geq k_0+m_a>n_b.$$ This is a contradiction, so $J_1\cdots J_{a}K_0$ is not standard.
			\item
			The terms with $k<n_b$ in Equation \eqref{eqn:relationterm61} vanish unless $h_b-i\leq  n_b-k$. In this case the lower $\bpartial$-lists of the terms are
			$$K_1= \bpartial^k(u^b_{h_b},\dots,u^b_{1} | v'_{s_1},\dots,v'_{s_i},\underline{v^b_{\sigma(i+1)},\dots,v^b_{\sigma(h_b)}}).$$
			By Lemma \ref{lemma:critgreat}, if there is  $E\in \cE(K_1)$, such that $E_{a}\leq E$ and $E_c\leq E$, we have 
			$wt(E(L))\geq k_0$ and 
			$$wt(E(R))-wt(E_{b-2})\geq L(E_{b-2},||E(R)||)).$$ By Lemma \ref{lemma:replace} and Equation (\ref{eqn:456}),
			$$
			 L(E_{b-2},||E(R)||))\geq L(E_{b-2},\bpartial^k|v'_{s_1},\dots,v'_{s_h}))-(h_b-i)\geq m_a-m_{b-2}-(h_b-i).$$
			Then $k=wt(E'(L)+wt(E'(R))\geq k_0+m_a-(h_b-i)>k$. This is a contradiction, so $J_1\cdots J_{b-2}K_1$ is not standard.
		\end{enumerate}
	
		\item
		If $n_b-k_0\geq m_a$, then $i_0>n_b-k_0-m_a\geq 0$. There is $i_0\leq i_1\leq h_b$, such that $v^b_{i_1}<v^a_{\sigma_a(i_1-i_0+1)}$.

		Assume $v'_{i_2}=v^a_{\sigma_a(i_1-i_0+1)}$, then $i_2\geq i_1-i_0+1$.
		By Equation (\ref{eqn:relation2+},
		$$0=\sum_{\substack{i_1>s\\ i_2-1\geq s}}\sum_{\sigma}\frac {(-1)^{s}\sign(\sigma)}{s!(i_2-1-s)!}\sum_{k=0}^{n}\epsilon a^{s}_{k} $$
		\begin{equation}\label{eqn:relationterms3+0}
		\left( \begin{array}{ccc}
		\bpartial^{m-k}&|&\underline{v^b_{1} ,\dots,v^b_{s} ,v'_{\sigma(s+1)},\dots, v'_{\sigma(i_2-1)},v'_{i_2},\dots,v'_{h}}\\
		\bpartial^k(u^b_{h_b},\dots, u^b_{1}&|&v'_{\sigma(1)},\dots,v'_{\sigma(s)},\underline{v^b_{s+1},\dots,v^b_{i_1}},v^b_{i_1+1},\dots,v^b_{h_b}  )
		\end{array}\right).
		\end{equation}
		Here  $a_k^{s}$ are integers and $a_{n_b-l}^{s}=\delta_{0,l}$ for $0 \leq l<i_1-s$, and
		$\sigma$ are permutations of $\{1,\dots,i_2-1\}$.
		
		$$0=\sum_{\substack{i\geq i_1> s\\ i_2-1\geq s}}\sum_{\sigma, \sigma'}\frac {(-1)^{i+s}\sign(\sigma')\sign(\sigma)}{(i-i_1)!(h_b-i)!s!(i_2-1-s)!}\sum_{k=0}^{n}\epsilon a^{i,s}_{k} $$
		
		\begin{equation}\label{eqn:relationterms3+1}
		\left( \begin{array}{ccc}
		\bpartial^{n-k}
		&|&\underline{v'_{h_b+1},\dots,v'_{h_a},v^b_{1},\dots,v^b_{i_1},v^b_{\sigma'(i_1+1)},\dots ,v^b_{\sigma'(i)}},\dots,v^b_{\sigma'(h_b)}) \\
		\bpartial^k( u^b_{h_b}, \dots, u^b_{1}&|&
		v'_{\sigma(1)}, \dots,v'_{\sigma(s)},\underline{ v'_{\sigma(s+1)}, \dots ,v'_{\sigma(i_2-1)}, v'_{i_2},\dots,v'_{h_b}} )
		\end{array}\right).
		\end{equation}
	Here  $a_k^{i,s}$ are integers, $a_{n_b-l}^{i,s}=\delta_{0,l}$ for $0 \leq l<i-s$, 
		$\sigma$ are permutations of $\{1,\dots,i_2-1\},$ and $\sigma'$ are permutations of $\{i_1+1,i_1+2,\dots, h_b\}$.

		If $i_2> i_1-i_0+1$, we use Equation (\ref{eqn:relationterms3+0}) and if $i_2=i_1-i_0+1$ we use Equation (\ref{eqn:relationterms3+1}).

		In the above equations:
		\begin{enumerate}
			\item
			All the terms with $k=n_b+1,\dots, n$ precede $J_a$ since the weight of upper $\bpartial$-list is $n-k<n_a$.
			\item
			The terms with $k=n_b$ are $J_aJ_b$, the terms with upper $\bpartial$-list preceding $J_a$ (the upper $\bpartial$-lists are the $\bpartial$-lists given
			by replacing some $v'_{i}$, $i\geq i_2$ in $J_a$ by some $v^b_{k}$, $k\leq i_1$),
			and the terms with the lower $\bpartial$-lists
			$$K_0=\bpartial^k(u^b_{h_b},\dots,u^b_{1}|v'_{\sigma(1)},\dots,v'_{\sigma(i_1)},v^b_{i_1+1},\dots,v^b_{h_b}).$$
			All of the other terms cancel.
			If $i_2-1<i_1$, the terms of the form $K_0$ do not appear in Equations \eqref{eqn:relationterms3+0} and \eqref{eqn:relationterms3+1}.
			Otherwise  $i_2-1\geq i_1$. In this case $i_2\geq i_1+1>i_1-i_0+1$, $E_{b-2}$ must be in $\cE_h^R$.
			
	If $E\in \cE(K_0)$, such that $E_{b-2}\leq E$ and $E_c\leq E$.
	Then $wt(E(L))\geq k_0$ and $wt(E(R))-wt(E_{b-2}(h_b))\geq R(E_{b-2}, ||E(R)||)$ by Lemma \ref{lemma:critgreat}. By Lemma \ref{lem:lrnumber}
	$$R(E_{b-2}, E(R))>R(E_{b-2}, ||E_a(h_b||))-i_1-(i_1-i_0+1).$$
	By Corollary \ref{cor:lrnumber1}, 
	\begin{equation}\label{eqn:567}
	R(E_{b-2}, ||E_a(h_b||))=m_a-m_{b-2}.
	\end{equation}
	Then $n_b=wt(E(L))+wt(E(R))\geq k_0+m_a+i_0>n_b$. This is a contradiction, so $J_1\cdots J_{b-2}K_0$ is not standard.
			\item If $i_2>i_1-i_0+1$,
			the terms with $k<n_b$  in  Equation (\ref{eqn:relationterms3+0}) vanish unless $i_1-s\leq  n_b-k$. In this case the lower $\bpartial$-lists of the terms are
			$$K_1=  \bpartial^k(u^b_{h_b},\dots, u^b_{1}|v'_{\sigma(1)},\dots,v'_{\sigma(s)},\underline{v^b_{s+1},\dots,v^b_{i_1}},v^b_{i_1+1},\dots,v^b_{h_b})$$
			Underlined elements can be any underlined elements in Equation (\ref{eqn:relationterms3+0}).
			
				If $E\in \cE(K_1)$, such that $E_{b-2}\leq E$ and $E_c\leq E$.
			Then $wt(E_b(L))\geq k_0$ and $wt(E(R))-wt(E_{b-2})\geq R(E_{b-2}, E(R))$ by Lemma \ref{lemma:critgreat}. By Lemma \ref{lem:lrnumber}
			$$R(E_{b-2}, E(R))>R(E_{b-2}, ||E_a(h_b||))+s-(i_1-i_0+1).$$
			By Equation (\ref{eqn:567}) and (\ref{eqn:numberi0-1}), then
            $$k=wt(E(L))+wt(E(R))\geq k_0+m_a+i_0+s-i_1>k.$$ This is a contradiction, so $J_1\cdots J_{b-2}K_1$ is not standard.
			\item If $i_2=i_1-i_0+1$, the terms with $k<n_b$ in Equation (\ref{eqn:relationterms3+1}) vanish unless $i-s\leq  n_b-k$. In this case the lower $\bpartial$-lists of the terms are
			$$K_1=    \bpartial^k( u^b_{h_b}, \dots u^b_{1}|
			v'_{\sigma(1)}, \dots,v'_{\sigma(s)},\underline{ v'_{\sigma(s+1)}, \dots ,v'_{\sigma(i_2-1)}, v'_{i_2},\dots,v'_{h_b}} )$$
			Underlined elements can be any underlined elements in Equation (\ref{eqn:relationterms3+1}). If there is $E\in \cE(K_1)$ such that 
			$E_c\leq E$ and $E_{b-2}\leq E$.
	Let
			$$K'_0=\bpartial^{k}|
			v'_{1},\dots, v'_{i_2-1},v'_{l_{i_2}}, \dots, v'_{l_{h_b}}).$$
			Here $i_2\leq l_{i_2}< l_{j_2+1}<\cdots< l_{h_b}\leq h_a$. We have
			$wt(E(L))\geq k_0$ and 
		$$R(E_{b-2},E(R))\geq R(E_{b-2},K'_0)-(i_2-1-s)-(i-i_1)$$
		for some $K_0'$ by Lemma \ref{lemma:replace} since in $K'_1$ the number of $v^b_i$ with $v^b_i>v'_{i_2}$ is at most $i-i_1$. By  Corollary \ref{cor:lrnumber1} and \ref{cor:lrnumber2},
		$$R(E_{b-2},K_0')\geq R(E_{b-2},||E_a(h_b)||)=m_a-m_{b-2}.$$ So by Lemma \ref{lemma:critgreat},
		$$wt(E(R))-wt(E_{b-2})\geq R(E_{b-2},E(R))\geq m_a-m_{b-2}+i_0-(i-s).$$
		$$k=wt(E(L))+wt(E(R))\geq k_0+m_a+i_0-(i-s)>k.$$ This is a contradiction, so $J_1\cdots J_{b-2}K_1$ is not standard.	
	\end{enumerate}
		
	\end{enumerate}
	\textsl{Case 4:}$J_a\notin J_h^L\cup J_h^R$. Let $E_1\cdots E_{a}$ be the standard ordered product of elements of $\cE_h^+$ corresponding to $J_1\cdots J_{a}$.
	Assume $J_i\in \cJ^L_h$ for $1\leq i\leq c$, $J_i\in \cJ^R_h$ for $c< i\leq d$ and $J_i\in\cJ_{h-1}$ for $d<i\leq a$.
	\begin{enumerate}
		\item if $d<a-1$, let $E=E_{a-1}$;
		\item if $d=a-1$ and $d>c>0$, let $E=F(E_c, E_d)$;
		\item if $d=a-1=c$, let $E=F(E_c,E_0)$ with $E_0=|(1,0),(2,0),\dots,(h,0))$;
		\item if $d=a-1$and $d>c=0$, Let $E=F(E_0,E_d)$ such that $E_0=((h,0),\dots,(2,0),(1,0)|$.
	\end{enumerate}
	Then by Remark \ref{rem:LRcompare}, $E_a$ is the largest element in $\cE(J_a)$ such that $E\leq E_a$.
	Since $J_1\cdots J_b$ is not standard, than $J_b$ is not greater than $E_a$, by Lemma \ref{lemma:straight0},  $J_aJ_b=\sum \tilde K_i \tilde f_i$ with $\tilde K_i\in \cJ$, $\tilde f_i\in \R$ such that $\tilde K_i$ is either smaller than $J_a$ or $\tilde K_i$ is not greater than $E$. If $sz(\tilde K_i)<h$, $\tilde K_i\in \cJ_{h-1}^+$. If $sz(\tilde K_i)\geq h$, $\tilde K_i=\sum \tilde K_{ij}\tilde g_j$ with $\tilde K_{ij}\in \cJ$ and $sz(\tilde K_{ij})=h$. By Equation (\ref{eqn:relationLR+}),  $\tilde K_{ij}=\sum J_k^L J_k^R$ with $J_k^L\in \cJ_h^L$ and $J_k^R\in \cJ_h^R$. So $J_aJ_b=\sum K_{i}f_i$ with $K_{ij}\in J_h^+$, such that $\ K_i$ is either smaller than $J_a$ or $ K_i$ is not greater than $E$.
\end{proof}


\begin{thebibliography}{ABKS}
\bibitem{B} B. Bhatt, \textit{Algebraization and Tannaka duality}, Camb. J. Math. 4 (2016), no. 4, 403-461. 
 	\bibitem{BS} D. Bourqui and J. Sebag, \textit{The radical of the differential ideal generated by $XY$ in the ring of two variable differential polynomials is not differentially finitely generated}, J. Commut. Algebra 11 (2019), no. 2, 155-162. 	
	\bibitem{DCP} C. de Concini and C. Procesi, \textit{A characteristic free approach to invariant theory}, Advances in Math. 21 (1976), no. 3, 330-354. 
	\bibitem{EM} L. Ein, M. Mustata, Jet schemes and Singularities, in: algebraic Geometry-Seattle 2005, Part 2, in: Proc. Sympos. Pure Math., vol. 80, Part2, Amer. Math. Soc., Providence, RI, 2009, pp. 505-546.
	\bibitem{FM} E. Feigin and I. Makedonskyi, \textit{Semi-infinite Pl\"ucker relations and Weyl modules}, Int. Math. Res. Not. IMRN 2020, no. 14, 4357-4394.
	\bibitem{Kol} E. Kolchin, \textit{Differential algebra and algebraic groups}, Academic Press, New York 1973.
	\bibitem{KS} K. Kpognon, J. Sebag, \textit{Nilpotency in arc schemes of plane curves}, Comm. in Algebra, Vol. 45 no 5 (2017), 2195-2221 
	\bibitem{LR} V. Lakshmibai, K. N. Raghavan, Standard Monomial Theory: Invariant Theoretic Approach. Encyclopaedia of Mathematical Sciences Volume 137. Springer-Verlag Berlin Heidelberg, 2008.
	
	\bibitem{LS} V. Lakshmibai and P. Shukla, Standard monomial bases and geometric consequences for certain rings of invariants. Proc. Indian Acad. Sci. Math. Sci. 116 (2006), no. 1, 9-36. 
	
	
	
	\bibitem{LS1} A. Linshaw and B.~Song, \textit{Standard monomials and invariant theory for arc space I: general linear group}, Commun. Contemp. Math. 26 no. 4 (2024),  2350013, 38pp.
	
	\bibitem{LS2} A. Linshaw and B.~Song, \textit{The global sections of chiral de Rham complexes on compact Ricci-flat K\"ahler manifolds II}, Comm. Math. Phys. 399 (2023), no. 1, 189-202. 
	
	\bibitem{LS3} A. Linshaw and B.~Song, \textit{Cosets of free field algebras via arc spaces}, Int. Math. Res. Not. vol. 2024 no. 1 (2024), 47-114.
	\bibitem{LSSI}A.~Linshaw, G.~Schwarz, B.~Song, \textit{Jet schemes and invariant theory}, Ann. I. Fourier 65  (2015) no. 6,  2571-2599.
	\bibitem{M} I. Makedonskyi, \textit{Semi-infinite Pl\"ucker relations and arcs over toric degeneration}, Math. Res. Lett. 29 (2022), no. 5, 1499-1536.
	 \bibitem{MSV}F.~Malikov, V.~Schechtman, and A.~Vaintrob, \textit{Chiral de Rham complex}, Comm. Math. Phys.  204, (1999) 439-473.
	 \bibitem{O} S. Odake, \textit{Extension of $N=2$ superconformal algebra and Calabi-Yau compactification}, 
Modern Phys. Lett. A 4 (1989), no. 6, 557-568.
	\bibitem{SI} J. Sebag, \textit{Arcs schemes, derivations and Lipman's theorem}, J. Algebra 347 (2011) 173-183.
	\bibitem{SII} J. Sebag, \textit{A remark on Berger's conjecture, Kolchin's theorem and arc schemes}, Archiv der Math., Vol. 108 no (2017), 145-150.
	\bibitem{W} H. Weyl, The Classical Groups. Their Invariants and Representations, Princeton University Press, Princeton, N.J., 1939.

	
	
	
\end{thebibliography}
\end{document}